\documentclass[final,a4paper,11pt]{article}
\usepackage{graphicx}
\usepackage{mathptmx}
\usepackage{amssymb,amsmath, amsfonts,amsthm}
\usepackage{a4wide}
\usepackage{color}

\usepackage{hyperref}

\newcommand{\email}[1]{{\tt #1}}
\newcommand{\R}{\mathbb{R}}
\newcommand{\oR}{\overline{\R}}

\newcommand{\norm}[1]{\|#1\|}

\newcommand{\dist}[1]{{\rm dist}(#1)}

\newcommand{\mv}{\,\mid\,}

\newcommand{\B}{{\cal B}}

\newcommand{\K}{{\cal K}}
\newcommand{\M}{{\cal M}}

\newcommand{\V}{{\cal V}}
\newcommand{\U}{{\cal U}}
\newcommand{\Sp}{{\mathcal S}}
\newcommand{\F}{{\cal F}}
\newcommand{\W}{{\cal W}}
\newcommand{\X}{{\cal X}}
\newcommand{\Y}{{\cal Y}}
\newcommand{\Z}{{\cal Z}}
\newcommand{\lag}{{\cal L}}
\newcommand{\setto}[1]{\mathop{\rightarrow}\limits^#1}

\newcommand{\longsetto}[1]{\mathop{\longrightarrow}\limits^{#1}}

\newcommand{\skalp}[1]{\langle #1\rangle}
\newcommand{\bskalp}[1]{\big\langle #1\big\rangle}

\newcommand{\xb}{{\bar x}}
\newcommand{\yb}{{\bar y}}
\newcommand{\zb}{{\bar z}}

\newcommand{\bb}{{\bar b}}
\newcommand{\AT}[2]{{\textstyle{#1\atop#2}}}
\newcommand{\xba}{{\bar x^\ast}}
\newcommand{\yba}{{\bar y^\ast}}
\newcommand{\zba}{{\bar z^\ast}}

\newcommand{\aba}{{\bar a^\ast}}
\newcommand{\oo}{o}
\newcommand{\OO}{{\cal O}}

\newcommand{\argmin}{\mathop{\rm arg\,min}}

\newcommand{\lin}{{\rm lin\,}}

\newcommand{\ri}{{\rm ri\,}}

\newcommand{\co}{{\rm conv\,}}
\newcommand{\gph}{\mathrm{gph}\,}

\newcommand{\dom}{\mathrm{dom}\,}

\newcommand{\tto}{\rightrightarrows}
\newcommand{\onabla}{\overline\nabla}
\newcommand{\Limsup}{\mathop{{\rm Lim}\,{\rm sup}}}
\newcommand{\Liminf}{\mathop{{\rm Lim}\,{\rm inf}}}
\newcommand{\Lim}{\mathop{{\rm Lim}}}
\newcommand{\myvec}[1]{\begin{pmatrix}#1\end{pmatrix}}

\newcommand{\prox}[1]{{\rm prox}_{#1}}

\newcommand{\rge}{{\rm rge\;}}
\renewcommand{\ker}{{\rm ker\;}}

\newcommand{\lsc}{l.s.c. }
\newcommand{\SKKT}{S_{\rm KKT}}

\newlength{\myAlgBox}
\newtheorem{theorem}{Theorem}[section]
\newtheorem{proposition}[theorem]{Proposition}
\newtheorem{remark}[theorem]{Remark}
\newtheorem{lemma}[theorem]{Lemma}

\newtheorem{definition}[theorem]{Definition}
\newtheorem{example}[theorem]{Example}

\newtheorem{assumption}{Assumption}

\title{Characterizations of the Aubin property of the KKT-mapping in composite optimization by SC derivatives and quadratic bundles}
\author{Helmut Gfrerer\thanks{Johann Radon Institute for Computational and Applied Mathematics (RICAM), A-4040 Linz, Austria and Institute of Information Theory and Automation, Czech Academy of
Sciences, 18208 Prague, Czech Republic; \email{helmut.gfrerer@ricam.oeaw.ac.at}}
\and   Ji\v{r}\'{i} V. Outrata\thanks{Institute of Information Theory and Automation, Czech Academy of
 Sciences, 18208 Prague, Czech Republic,  \email{outrata@utia.cas.cz}}
}

\date{}

\begin{document}
\maketitle
\begin{abstract}
  For general set-valued mappings, the Aubin property is ultimately tied to limiting coderivatives by the Mordukhovich criterion. Likewise, the existence of single-valued Lipschitzian  localizations is related to strict graphical derivatives. In this paper we will show that for the special case of the KKT-mapping from composite optimization, the Aubin property and the existence of single-valued Lipschitzian localizations can be characterized by SC derivatives and quadratic bundles, respectively, which are easier accessible than limiting coderivatives and strict graphical derivatives.
\end{abstract}

\section{Introduction}
Consider the composite optimization problem
\begin{equation}\label{EqOptProbl}\min \varphi(x):=f(x)+g(F(x)),\end{equation}
where $f:\R^n\to\R$ and $F:\R^n\to\R^m$ are twice continuously differentiable and $g:\R^m\to\oR:=\R\cup\{\infty\}$ is a  lower semicontinuous (\lsc) function satisfying $\dom g:=\{y\in\R^m\mv g(y)<\infty\}\not =\emptyset$.

The problem \eqref{EqOptProbl} is a rather general problem. If one takes $g$ to be the indicator function $\delta_K$ of a closed convex set $K$, the term $g(F(x))$ corresponds to a constraint $F(x)\in K$. In particular, in the special case
\[K=\{y\in\R^m\mv y_i=0,\ i=1,\ldots,l,\ y_i\leq 0,\ i=l+1,\ldots,m\}\]
we recover the {\em nonlinear programming problem}. Other instances of \eqref{EqOptProbl} are {\em second-order cone programming} and {semidefinite programming} when $K$ is the second-order cone and the cone of positive semidefinite matrices, respectively. In this paper we do not assume such special structures and deal with very general \lsc convex functions $g$.

Given a local minimizer $\xb$ of \eqref{EqOptProbl}, the first-order optimality conditions (KKT-conditions) read as
\[0=\nabla_x\lag(x,y^*),\ y^*\in \partial g(F(x)),\]
where the Lagrangian $\lag:\R^n\times\R^m\to\R$ is defined by
\[\lag(x,y^*):=f(x)+\skalp{y^*,F(x)},\ (x,y^*)\in\R^n\times\R^m.\]
We consider the {\em canonically perturbed} version of \eqref{EqOptProbl} given by
\begin{equation}\label{EqPertOptProbl}
  \min f(x)-\skalp{a^*,x}+g(F(x)+b)
\end{equation}
with perturbation parameters $a^*\in\R^n$ and $b\in\R^m$. The corresponding KKT system reads as
\begin{equation}\label{EqKKTpert}
0= \nabla_x\lag(x,y^*)-a^*, y^*\in\partial g(F(x)+b)
\end{equation}
and we are interested in properties of its solution mapping
\begin{equation}
  \SKKT(a^*,b):=\{(x,y^*)\in\R^n\times\R^m\mv (x,y^*,a^*,b) \mbox{ fulfills }\eqref{EqKKTpert}\}
\end{equation}
for perturbation parameters $(a^*,b)$ close to the reference parameter
\begin{equation}
  (\aba,\bb):=(0,0).
\end{equation}
Various stability properties for solution mappings of parameter dependent generalized equations and optimization problems have been considered in the literature, we refer her only to the monographs \cite{BonSh00, DoRo14, KlKum02, Mo06a, Mo06b, Mo24} and the references therein.
 In this paper we pay attention to the so-called Aubin property and the existence of single valued Lipschitzian localizations of the KKT-mapping $\SKKT$. In particular, we want to analyze {\em full primal-dual stability} of a pair $(\xb,\yba)\in\SKKT(\aba,\bb)$ as recently introduced by Benko and Rockafellar \cite{BeRo24}. Let us define the following mappings associated with optimal solutions of the perturbed optimization problem \eqref{EqPertOptProbl} by
\begin{align}
  &S_{\rm Opt}^\delta(a^*,b):=\argmin_{\norm{x-\xb}\leq\delta}\big\{f(x)-\skalp{a^*,x}+g(F(x)+b)\},\\
  &\widehat S_{\rm Opt}^\delta(a^*,b):=\{(x,y^*)\in \SKKT(a^*,b)\mv x\in S_{\rm Opt}^\delta(a^*,b), \norm{y^*-\yba}\leq\delta\}.
\end{align}
\begin{definition}\label{DefPrimalDualStability}
  The primal--dual pair $(\xb,\yba)$ is {\em fully stable}
in problem \eqref{EqOptProbl} if there is a neighborhood $\U^*\times\V$ of $(\aba,\bb)$ such that, for $\delta>0$ sufficiently small,
the mapping $\widehat S_{\rm Opt}^\delta$ is single-valued and Lipschitz continuous on $\U^*\times\V$.
\end{definition}
At the very general level, characterization of stability properties of set-valued mappings can be stated by generalized derivatives.
By the celebrated Mordukhovich criterion, the Aubin property   can be smartly characterized in terms of limiting coderivatives. Though a lot of calculus rules for coderivatives are available, it can be hard or even an unsurmountable hurdle to compute the exact limiting coderivative. Often the calculus rules only yield an inclusion resulting in something which is only sufficient for the Aubin property but not necessary.

Likewise, the existence of a single-valued Lipschitzian localization is associated with strict graphical derivatives as demonstrated by Kummer's inverse mapping theorem \cite{Kum91}, see also \cite[Theorem 9.54]{RoWe98}, \cite[Section 3]{Ro25_prepr} for the general case. Similarly to limiting coderivatives, the calculation of strict graphical derivatives is an extremely difficult task which impairs their usefulness.

The objective of this paper is to show that the Aubin property of  $\SKKT$ and full primal-dual stability, respectively, can be characterized by {\em subspace containing  derivatives (SC derivatives)} of the subdifferential mapping $\partial g$ and Rockafellar's {\em quadratic bundle} \cite{Ro23a} for $g$, respectively, which are easier accessible than limiting coderivatives and strict graphical derivatives as discussed  at the end of Subsection \ref{SubsecSCD_quad}.
SC derivatives were  introduced very recently by the authors \cite{GfrOut22} in order to provide a theoretical framework for the efficient implementation of the semismooth* Newton method from \cite{GfrOut21}. Apart from this numerically  motivated approach, it already turned out in \cite{GfrOut22, GfrOut23} that SC derivatives are a useful tool for analyzing stability properties of set-valued mappings. They  represent a first-order generalized derivative for a certain class of set-valued mappings including the subdifferential mapping of \lsc prox-regular and subdifferentially continuous functions.
On the other hand, quadratic bundles are based on epigraphical limits of certain second subderivatives of such functions. Both notions are linked by a one-to-one correspondence and therefore all results can be  formulated either by SC derivatives or by quadratic bundles in an equivalent way.

At the core of our analysis there is a result that a certain second order qualification condition \cite[Equation (3.15)]{MoRo12}, expressed in terms of limiting coderivatives, can be equivalently formulated using SC derivatives, quadratic bundles and strict graphical derivatives, respectively. Then we can show, under the assumption that the SC derivative of $\partial g$ is a singleton, that the Aubin property of $\SKKT$ is  equivalent to the existence of a single-valued Lipschitzian localization as well as to the simultaneous fulfillment of the second order qualification condition and another second-order condition involving the SC derivative. Moreover, we can show that the single-valued localization is not only Lipschitzian but also strictly differentiable and compute its derivative. These results are based on characterizations of strict proto-differentiability from the very recent works \cite{Gfr25c,GfrOut25_prepr}.

In case of local minimizers we will show, under the assumption that a condition necessary for variational sufficiency is fulfilled, that the Aubin property of $\SKKT$ is  equivalent to the existence of a single-valued Lipschitzian localization as well as to strong variational sufficiency and full primal-dual stability. These properties  can be characterized by the second-order qualification condition and a condition on SC derivatives/quadratic bundles ensuring strong variational sufficiency. Here we apply results from Benko and Rockafellar \cite{BeRo24} and Rockafellar's characterization of strong variational sufficiency via quadratic bundles \cite{Ro23a}. Finally,  under the assumption that a certain chain rule is fulfilled, we will show that in the characterizations above the requirement of strong variational sufficiency can be replaced by the property that the reference point $\xb$ is a tilt-stable local minimizer.

The paper is organized as follows. In Section 2 we recall all the basic definitions and properties used in this paper. Section 3 is devoted to the second-order qualification condition and its consequences. Finally, in Section 4 we state our results on the Aubin property of $\SKKT$ and the full primal-dual stability.

The following notation is employed. For an element $x\in\R^n$, $\norm{x}$ denotes its Euclidean norm and $\B(x,\delta)$ denotes the closed ball around $x$ with radius $\delta$, whereas $\B$ stands for the closed unit ball. In a product space we use the norm $\norm{(u,v)}:=\sqrt{\norm{u}^2+\norm{v}^2}$. Given
an $m\times n$ matrix $A$, we employ the operator norm $\norm{A}$ with respect to the Euclidean norm and we denote the range of $A$ by $\rge A$ and the Moore-Penrose inverse of $A$ by $A^\dag$. For two $m\times n$ matrices $A,B$, we set $\rge(A,B)=\{(Ap,Bp)\mv p\in\R^n\}\subset\R^m\times\R^m$. When a function $f:\R^n\to\R$ is differentiable at $x\in\R^n$ we denote by $\nabla f(x)$ its gradient and for a mapping $F:\R^n\to\R^m$ the notation $\nabla F(x)$ is used to denote its Jacobian. For twice differentiable functions $f:\R^n\to \R$ we denote by $\nabla^2f(x)$ the Hessian of $f$ at $x$.

 Given a linear subspace $L\subseteq \R^n$, $ L^\perp$ denotes its orthogonal
complement and, for a closed cone $K$ with vertex at the origin, $K^\circ$ signifies its (negative) polar.
Further, given a multifunction $F:\R^n\tto\R^m$, $\gph F:=\{(x,y)\mv y\in F(x)\}$ stands for its
graph, its inverse is given by $F^{-1}(y):=\{x\in\R^n\mv y\in F(x)\}$, $y\in\R^m$ and $\ker F$ denotes the set $\ker F:=\{x\mv 0\in F(x)\}$.   Given a set $\Omega\subset\R^s$, we define the distance of a point $x$ to $\Omega$ by $d_\Omega(x):=\dist{x,\Omega}:=\inf\{\norm{y-x}\mv y\in\Omega\}$ and the indicator function is denoted by $\delta_\Omega$. The notation $x_k\setto{\Omega} \xb$ means that the sequence $x_k$ converges to $\xb$ and $x_k\in\Omega$ $\forall k$.

\section{Preliminaries}
At the beginning let us recall set convergence in the sense of Painlev\'e--Kuratowski.
Given a parameterized family $C_t$ of subsets of a metric space, where $t$ also belongs to a metric space, the upper (outer) and lower(inner) limits are given by
\begin{align*}
\Limsup_{t\to\bar t}C_t:=\{x\mv \liminf_{t\to\bar t}\,\dist{x,C_t}=0\},\
\Liminf_{t\to\bar t}C_t:=\{x\mv \limsup_{t\to\bar t}\,\dist{x,C_t}=0\}.
\end{align*}
The limit of the sequence exists if the outer and inner limit sets are equal and we set
\[\Lim_{t\to\bar t}C_t:=\Limsup_{t\to\bar t}C_t=\Liminf_{t\to\bar t}C_t.\]
\subsection{Subdifferentials and second subderivatives}
Given an extended real-valued function $\psi:\R^n\to\oR$ and  a point $\xb\in \dom  \psi$, 
the {\em regular subdifferential} of $\psi$ at $\xb$ is given by
\[\widehat\partial  \psi(\xb):=\Big\{x^*\in\R^n\mv\liminf_{x\to\xb}\frac{ \psi(x)- \psi(\xb)-\skalp{x^*,x-\xb}}{\norm{x-\xb}}\geq 0\Big\},\]
while the {\em (limiting) subdifferential} is defined by
\[\partial  \psi(\xb):=\{x^*\mv \exists x_k\to \xb, x_k^* \to x^* \mbox{ with }\psi(x_k)\to\psi(x)\mbox{ and } x_k^*\in\widehat \partial  \psi(x_k)\ \forall k\}.\]
The {\em singular (horizon) subgradient set} is given by
\[\partial^\infty\psi(x)=\{x^*\mv \exists t_k\downarrow 0,x_k\to \xb, x_k^* \in\widehat\partial \psi(x_k) \mbox{ with }\psi(x_k)\to\psi(x)\mbox{ and } t_kx_k^*\to x^*\}.\]
An \lsc function $\psi:\R^n\to\oR$ is called {\em prox-regular} at $\xb$ for $\xba$
 if $\psi$ is finite  at $\xb$ with $\xba\in\partial \psi(\xb)$ and  there
exist $\epsilon> 0$ and $r\geq 0$ such that
\[ \psi(x')\geq  \psi(x)+\skalp{x^*,x'-x}-\frac r2 \norm{x'-x}^2\]
 whenever $x',x\in \B(\xb,\epsilon)$, $x^*\in \B(\xba,\epsilon)$ and $\psi(x)<\psi(\xb)+\epsilon$. When this holds for all $\xba\in\partial\psi(\xb)$, $\psi$ is called to be prox-regular at $\xb$.

Further, $\psi$ is called {\em sub\-differen\-tially continuous} at $\xb$ for $\xba$ if for any sequence $(x_k,x_k^*)\longsetto{{\gph \partial  \psi}}(\xb,\xba)$ we have $\lim_{k\to\infty} \psi(x_k)= \psi(\xb)$. When this holds for all $\xba\in\partial\psi(\xb)$, $\psi$ is called to be subdifferentially continuous at $\xb$.

Recall that an \lsc convex function is prox-regular and subdifferentially continuous on its whole domain.

Define the parametric family of second-order difference quotients for $\psi$ at $\xb$ for $\xba\in\R^n$ by
\[\Delta_t^2\psi(\xb,\xba)(u):=\frac{\psi(\xb+tu)-\psi(\xb)-t\skalp{\xba,u}}{\frac 12 t^2}\quad\mbox{ with }u\in\R^n,\ t>0.\]
If $\psi(\xb)$ is finite, then the {\em second subderivative} of $\psi$ at $\xb$ for $\xba$ is given by
\[{\rm d^2}\psi(\xb,\xba)(u)=\liminf_{\AT{t\downarrow 0}{u'\to u}}\Delta_t^2\psi(\xb,\xba)(u').\]
$\psi$ is called {\em twice epi-differentiable} at $\xb$ for $\xba$, if the functions $\Delta_t^2\psi(\xb,\xba)$ epi-converge to ${\rm d^2}\psi(\xb,\xba)$ as $t\downarrow 0$.

\subsection{Set-valued mappings and generalized differentiation}
Consider a set $\Omega\subset\R^n$ and a point $\zb\in\Omega$.
 The {\em tangent cone} and the {\em paratingent cone}  to $\Omega$ at $\zb$ are given by
\begin{gather*}
T_\Omega(\zb):=\Limsup_{t\downarrow 0}\frac{\Omega-\zb}t,\quad
T^P_\Omega(\zb):= \Limsup_{\AT{z\setto{\Omega}\zb}{t\downarrow 0}}\frac{\Omega-z}t
\end{gather*}
and the {\em regular normal cone} and the {\em limiting normal cone} to $\Omega$ at $\zb$ are defined as 
\begin{gather*}
\widehat N_\Omega(\zb):=\big(T_\Omega(\zb)\big)^\circ,\quad
N_\Omega(\zb):=\Limsup_{z\setto{\Omega}\zb}\widehat N_\Omega(z).
\end{gather*}

We now proceed with generalized derivatives of set-valued mappings.
Let $F:\R^n\tto\R^m$ be a mapping and let $(\xb,\yb)\in\gph F$.
 The {\em strict derivative} $D_*F(\xb,\yb):\R^n\tto\R^m$, the {\em graphical derivative} $DF(\xb,\yb):\R^n\tto\R^m$ and the {\em limiting (Mordukhovich) coderivative} $D^*F(\xb,\yb):\R^n\tto\R^m$ at $\xb$ for $\yb$ are given by
        \begin{gather*}\gph D_*F(\xb,\yb)=T_{\gph F}^P(\xb,\yb),\quad \gph DF(\xb,\yb)=T_{\gph F}(\xb,\yb),\\
         \gph D^*F(\xb,\yb)=\{(y^*,x^*)\mv (x^*,-y^*)\in N_{\gph F}(\xb,\yb)\}.
         \end{gather*}

These generalized derivatives may be used to characterize several regularity properties of set-valued mappings.   Given $F:\R^n\tto\R^m$ and a point $(\xb,\yb)\in\gph F$, the mapping  $F$ is said to be {\em metrically regular around} $(\xb,\yb)$ if there is $\kappa\geq 0$ together with neighborhoods $\X$ of $\xb$ and $\Y$ of $\yb$ such that
      \begin{equation}\label{EqMetrReg}
      \dist{x,F^{-1}(y)}\leq \kappa\dist{y,F(x)}\ \forall (x,y)\in \X\times \Y.
    \end{equation}
$F$ is said to be {\em strongly metrically regular around} $(\xb,\yb)$ if it is metrically regular around $(\xb,\yb)$ and $F^{-1}$ has a single-valued graphical localization around $(\yb,\xb)$, i.e., there are  open neighborhoods $\Y'$ of $\yb$, $\X'$ of $\xb$ and a mapping $h:\Y'\to\R^n$ with $h(\yb)=\xb$ such that $\gph F\cap (\X'\times \Y')=\{(h(y),y)\mv y\in \Y'\}$.

It is well-known that $F$ is metrically regular around $(\xb,\yb)$ if and only if the inverse mapping $S=F^{-1}$  has the so-called {\em Aubin property} around $(\yb,\xb)$, i.e., there is $\kappa\geq 0$ together with neighborhoods $\Y$ of $\yb$ and $\X$ of $\xb$ such that
      \begin{equation}\label{EqAubin}
      S(y)\cap \X\subset \kappa\norm{y-y'}\B +S(y')\ \forall y,y'\in \Y.
    \end{equation}
For single-valued mappings $S:\R^m\to\R^n$, the Aubin property of $S$ around $(\yb,S(\yb))$ is the same as Lipschitz continuity of $S$ near $\yb$.
Therefore, $F$ is strongly metrically regular around $(\xb,\yb)$ if and only if $F^{-1}$ has a single-valued graphical localization around $(\yb,\yb)$ which is Lipschitz continuous.

In this paper we will use the following characterizations of metric regularity and the existence of single-valued Lipschitzian localizations.
\begin{theorem}\label{ThMordRoCrit}
  Consider a multifunction  $F:\R^n\tto\R^m$ and assume that $\gph F$ is locally closed around $(\xb,\yb)\in\gph F$.
  \begin{enumerate}
    \item[(i)] (Mordukhovich criterion, cf. \cite[Theorem 3.3]{Mo18}) $F$ is metrically regular around $(\xb,\yb)$ if and only if
      \begin{equation}
            \label{EqMoCrit} 0\in D^*F(\xb,\yb)(y^*)\ \Rightarrow\ y^*=0.
      \end{equation}
    \item[(ii)] (cf. Rockafellar \cite[Proposition 3.1]{Ro25_prepr}.) $F^{-1}$ has a graphical localization around $(\yb,\xb)$ which is single-valued and Lipschitz continuous relative to its domain if and only if
        \begin{equation}\label{EqRoStrictDerivCrit}D_*F(\xb,\yb)^{-1}(0)=\{0\}.\end{equation}
    \end{enumerate}
\end{theorem}
Note that the strict derivative criterion \eqref{EqRoStrictDerivCrit} can be also equivalently written as
\[\ker D_*F(\xb,\yb)=\{0\}\quad\mbox{or}\quad \Big(0\in D_*F(\xb,\yb)(u)\ \Rightarrow\ u=0\Big).\]

We now turn our attention to another notion of differentiability. We call a set-valued mapping $F:\R^n\tto\R^m$ {\em strictly proto-differentiable} at $(\xb,\yb)\in\gph F$ if
\[\Lim_{\AT{(x,y)\longsetto{\gph F}(\xb,\yb)}{t\downarrow 0}}\frac{\gph F-(x,y)}t\]
exists. This definition goes back to the work of Poliquin and Rockafellar \cite{PolRo96b} and has its roots in the notion of {\em strictly smooth sets} introduced by Rockafellar \cite{Ro85}. One of the most important features of strict proto-differentiability is that it provides a linkage between the limiting coderivative and the strict graphical derivative: If $\gph F$ is locally closed at $(\xb,\yb)$ then  $F$ is strictly proto-differentiable at $\xb$ for $\yb$ if and only if both $\gph D_*F(\xb,\yb)$ and $\gph D^*F(\xb,\yb)$ are subspaces satisfying
  \begin{equation}\label{EqLemStrProto}\gph D^*F(\xb,\yb)=\{(y^*,x^*)\mv (x^*,-y^*)\in\gph D_*F(\xb,\yb)^\perp\},\end{equation}
  cf. \cite[Lemma 2.4]{Gfr25c}. This has also the consequence \cite[Theorem 3.6]{Gfr25c} that metric regularity of $F$ around $(\xb,\yb)$ implies strong metric regularity, provided $F$ is graphically Lipschitzian of dimension $m$ around $(\xb,\yb)$ according to the following definition.
  \begin{definition}[cf.{\cite[Definition 9.66]{RoWe98}}]\label{DefGraphLip}A mapping $F:\R^n\tto\R^m$ is {\em graphically Lipschitzian of dimension $d$} at $(\xb,\yb)\in\gph F$ if there is an open neighborhood $W$ of $(\xb,\yb)$ and a one-to-one mapping $\Phi$ from $W$ onto an open subset of $\R^{n+m}$ with $\Phi$ and $\Phi^{-1}$ continuously differentiable, such that $\Phi(\gph F\cap W)$  can be identified with the graph of a Lipschitz continuous mapping $f:U\to\R^{n+m-d}$, where $U$ is an open set in $\R^d$.
\end{definition}
A prominent example for a graphically Lipschitzian mapping is the subdifferential mapping of an \lsc convex function $g:\R^n\to\oR$. Consider the {\em proximal mapping} $\prox g$ of $g$ defined by
\[\prox g(z):=\argmin_x\{\frac 12\norm{x-z}^2+g(x)\}\]
Then it is well-known by Minty's Theorem, see, e.g., \cite[Theorem 12.15]{RoWe98}, that $\prox g$ is single-valued and Lipschitzian, in fact nonexpansive, on $\R^m$ and $\gph \prox g=\Phi(\gph \partial g)$ with $\Phi(x,x^*):=(x+x^*,x)$.

At the end of this subsection let us recall the notion of the {\em B-differential} (B-Jacobian) of a single-valued mapping $F:U\to\R^m$, where $U\subset\R^n$ is an  open set. The B-differential of $F$ at $x \in U$ is defined as
\[\onabla F(x):=\{A\in\R^{m\times n}\mv \exists x_k\to x: \mbox{ $F$ is Fr\'echet differentiable at $x_k$ and }A = \lim_{k\to\infty}\nabla F(x_k)\}.\]
Recall that the Clarke Generalized Jacobian is given by $\co\onabla F(x)$, i.e., the convex hull of the B-differential.

\subsection{\label{SubsecSCD_quad} On SC derivatives of the subdifferential mapping and quadratic bundles}
Let us briefly recall some basics about SC derivatives which are used in this paper.
Consider the metric space $\Z_n$ of all $n$-dimensional subspaces of $\R^n\times \R^n$ equipped with the metric
\[d_\Z(L_1,L_2)=\norm{P_{L_1}-P_{L_2}},\]
where $P_{L_i}$, $i=1,2$, denotes the orthogonal projection onto $L_i$. Given a subspace $L\in  \Z_n$, we denote by
\[L^*:=\{(y^*,x^*)\in\R^m\times\R^n\mv (x^*,-y^*)\in L^\perp\}.\]
its {\em adjoint }subspace. Then $(L^*)^*=L$ and $d_\Z(L_1,L_2)=d_\Z(L_1^*,L_2^*)$, cf. \cite{GfrOut22}. Since $\dim L^*=\dim L^\perp = n+n-\dim L=n$, we have $L^*\in\Z_{n}$ whenever $L\in\Z_{n}$.

\begin{definition}\label{DefSCD}Let $F:\R^n\tto\R^n$ be a mapping.
  \begin{enumerate}
    \item $F$ is called {\em graphically smooth} at $(x,y)\in\gph F$ and of dimension $d$ in this respect, if $T_{\gph F}(x,y)$ is a $d$-dimensional subspace of $\R^n\times\R^n$. We denote by $\OO_F$ the set of all points from the graph of $F$, where $F$ is graphically smooth of dimension $n$.
    \item The  {\em subspace containing derivative} (SC derivative) $\Sp F:\gph F\tto  \Z_n$ is defined by
    \begin{gather*}
      \Sp F(x,y):=\{L\in \Z_{n}\mv \exists (x_k,y_k)\longsetto {\OO_F}(x,y): d_\Z(T_{\gph F}(x_k,y_k),L)=0\},\\
    \end{gather*}
    whereas the {\em adjoint} SC derivative   $\Sp^*F:\gph F\tto \Z_n$ is given by
    \[\Sp^*F(x,y):=\{L^*\mv L\in\Sp F(x,y)\}.\]
    \item We say that $F$ has the {\em  SCD (subspace containing derivative) property} at $(\xb,\yb)\in\gph F$  if $\Sp F(\xb,\yb)\not=\emptyset$.
        Further we say that $F$ is an  SCD mapping if it has the SCD property at every $(x,y)\in\gph F$.
  \end{enumerate}
\end{definition}

\begin{remark}
  The SC derivatives $\Sp F$ and $\Sp^*F$ were introduced in \cite{GfrOut22} in a slightly different but equivalent way. We refer to \cite{GfrOut23} for SC derivatives of mappings $F:\R^n\tto\R^m$.
\end{remark}

By combining Lemmas 3.7, 3.10 and Remark 3.9 from \cite{GfrOut22} we obtain
\begin{equation}\label{EqIncl1}
  L\subset \gph D_*F(x,y),\ L^*\subset \gph D^*F(x,y)\quad \mbox{for all }L\in\Sp F(x,y).
\end{equation}
Hence, the SC derivatives can be interpreted as a kind of skeleton for the strict graphical derivative and the limiting coderivative, respectively.

We now turn our attention to the SC derivative of the subgradient mapping of a prox-regular and subdifferentially continuous function $\psi:\R^n\to\oR$.
\begin{proposition}[{\cite[Proposition 3.26]{GfrOut22}}]\label{PropSCDproxreg}
    Suppose that $\psi:\R^n\to\oR$ is prox-regular and subdifferentially continuous at $\xb$ for $\xba\in\partial\psi(\xb)$. Then for all $(x,x^*)\in\gph\partial \psi$ sufficiently close to $(\xb,\xba)$ one has $\Sp(\partial\psi)(x,x^*)=\Sp^*(\partial\psi)(x,x^*)\not=\emptyset$ and
    $L=L^*$ $\forall L\in \Sp(\partial\psi)(x,x^*)$, i.e., $L$ is self-adjoint.
\end{proposition}
For the practical computation of SC derivatives one needs a basis representation of the subspaces contained in it and
our practical experience tells us that the use of the following bases is advantageous. Let $\Z_n^{P,W}$ denote the collection of all subspaces $L\in\Z_n$ such that there are symmetric
$n\times n$ matrices $P$ and $W$ fulfilling $L = \rge(P,W)$ and
\begin{align}\label{EqPW}
 P^2 = P,\ W(I-P)=I-P^.
\end{align}
This implies that $P$ is the orthogonal projection onto some subspace of $\R^n$ and
\begin{equation}
  \label{EqPW1} W=PWP+(I-P),
\end{equation}
cf. \cite[Equation (3.11)]{Gfr25a}.
For any $L\in\Z_n^{P,W}$ the corresponding matrices $P,W$ are unique and, in the setting of Proposition \ref{PropSCDproxreg},  there holds $\Sp(\partial \psi)(x,x^*)\subset \Z_n^{P,W}$ for all $(x,x^*)\in \gph\partial \psi$ sufficiently close to $(\xb,\xba)$, cf. \cite{Gfr25a}. In what follows we denote by $\M_{P,W}\partial \psi(x,x^*)$ the collection of all symmetric $n\times n$ matrices $P,W$ fulfilling \eqref{EqPW} and
\begin{equation}\label{EqM_PW}
  \rge(P,W)\in\Sp(\partial \psi)(x,x^*).
\end{equation}
We now turn our attention to quadratic bundles introduced by Rockafellar \cite{Ro23a}. We restrict here the  definition to prox-regular and subdifferentially continuous functions, because for more general functions the definition has to be slightly changed in order to be useful,
 cf. \cite[Definition 4.4]{Ro25_prepr}.
 \begin{definition}\label{DefQuadrBun}
   \begin{enumerate}
     \item A function $q:\R^n\to\oR$ is called a {\em generalized quadratic form}, if $q(0)=0$ and the subgradient mapping $\partial g $ is {\em generalized linear}, i.e., $\gph \partial q$ is a subspace of $\R^n\times\R^n$.
     \item A function $\psi:\R^n\to\oR$ will be called {\em generalized twice differentiable} at $x$ for a subgradient $x^*\in\partial \psi(x)$, if it is twice epi-differentiable at $x$ for $x^*$ with the second subderivative ${\rm d}^2\psi(x, x^*)$ being a generalized quadratic form $q$.
     \item Given an \lsc function $\psi:\R^n\to\oR$ which is prox-regular and subdifferentially continuous at $x$ for $x^*\in\partial\psi(x)$, the {\em quadratic bundle} of $\psi$ at $x$ for $x^*$ is defined by
     \[{\rm quad\,}\psi(x,x^*):=\left[\ \begin{minipage}{11cm}the collection of generalized quadratic forms $q$ for which $\exists (x_k,x_k^*)\to(x,x^*)$ with $\psi$
             generalized twice differentiable at $x_k$ for $x_k^*$ and such that the generalized quadratic forms $q_k=\frac 12{\rm d}^2\psi(x_k, x_k^*)$ converge epigraphically to $q$.
     \end{minipage}\right.\]
   \end{enumerate}
 \end{definition}
 By taking into account that the generalized quadratic forms $q\in {\rm quad\,}g(x,x^*)$ in Definition \ref{DefQuadrBun} differ from the ones in \cite[Definition 3.30]{GfrOut22} by a factor $\frac 12$, the next statement follows from \cite[Proposition 3.33]{GfrOut22}.
 \begin{proposition}\label{PropSCD_quad}
 In the setting of Proposition \ref{PropSCDproxreg}, for every pair $(x,x^*)\in\gph\partial g$ sufficiently close to $(\xb,\xba)$ there holds
 \[\Sp(\partial \psi)(x,x^*)= \Sp^*(\partial \psi)(x,x^*) =\{\gph \partial q\mv q\in{\rm quad\,}\psi(x,x^*)\}.\]
 \end{proposition}
 For any pair $(P,W)$ of symmetric $n\times n$ matrices such that $P^2=P$ we can define a generalized quadratic form $q_{P,W}:\R^n\to\oR$ by
 \begin{equation}\label{EqQuad_PW}
   q_{P,W}(u):=\begin{cases}\frac 12\skalp{u,Wu}&\mbox{if $u\in\rge P$,}\\
   \infty&\mbox{otherwise.}
   \end{cases}
   \end{equation}
   Then we have $\gph \partial q_{P,W}= \rge(P, PWP+(I-P))$ and, by taking into account \eqref{EqPW1} and using the fact that a generalized quadratic form is uniquely given by its subdifferential, we conclude that
   \begin{equation}\label{EqQuad_PW_Equiv}
   {\rm quad\,}\psi(x,x^*) =\{q_{P,W}\mv (P,W)\in\M_{P,W}\partial \psi(x,x^*)\}.
 \end{equation}
At the end of this subsection, we recall some properties of the SC derivative of an \lsc convex function $g:\R^n\to\oR$.
Since $\gph \prox g= \Phi(\gph \partial g)$ with $\Phi(x,x^*)=(x+x^*,x)$, we obtain from \cite[Proposition 3.17]{GfrOut22} the formula
\begin{equation}\label{EqSCDconvex}\Sp \partial g(x,x^*)=\{\rge(B,I-B)\mv B\in \onabla \prox g(x+x^*)\}.
\end{equation}
Together with \cite[Corollary 3.28]{GfrOut22} we conclude that every matrix  $B\in \onabla \prox g(x+x^*)$ is symmetric positive semidefinte and satisfies $\norm{B}\leq 1$.
Further, by \cite[Proposition 7.1]{GfrOut22} together with \eqref{EqIncl1} we may formulate the following theorem, which will play an important role in the subsequent analysis.
\begin{theorem}
Given an \lsc convex function $g:\R^n\to\oR$, for every pair $(x,x^*)\in\gph \partial g$ there hold the inclusions
\begin{align}
\label{EqInclCoderiv}  \bigcup_{B\in \onabla \prox g(x+x^*)}\rge(B,I-B)\subset \gph D^*(\partial g)(x,x^*)\subset \bigcup_{B\in \co(\onabla \prox g(x+x^*))}\rge(B,I-B),\\
\label{EqInclStrictDeriv}  \bigcup_{B\in \onabla \prox g(x+x^*)}\rge(B,I-B)\subset \gph D_*(\partial g)(x,x^*)\subset \bigcup_{B\in \co(\onabla \prox g(x+x^*))}\rge(B,I-B),
\end{align}
\end{theorem}
From \eqref{EqSCDconvex} we may conclude  that computing $\Sp(\partial g)$ for an \lsc convex function $g$ is not more difficult than computing the B-differential of its proximal mapping. This is still a difficult task but it seems to be more tractable than computing limiting coderivatives or strict graphical derivatives. Note that the computation of the $(P,W)$-basis representation of the SC derivative is sometimes even simpler and is also applicable when the proximal mapping is not at hand but the subdifferential $\partial g$ is available. Determining the set $\OO_{\partial g}$ is usually not so difficult. In many cases it is  related with the set of pairs $(x,x^*)\in \gph\partial g$ fulfilling $x^*\in\ri\partial g(x)$. Also the calculation of tangent spaces at those points is  often not very complicated and only computing the possible limits is sometimes involved.
Further, consider, e.g., the special case when $g$ is twice continuously differentiable. Then $\M_{P,W}(\partial g)(x,\nabla g(x))=\{(I,\nabla^2 g(x))\}$, whereas $\onabla \prox g(x+\nabla g(x))=\{(I+\nabla^2 g(x))^{-1}\}$. This trivial example indicates that the computation of the $(P,W)$-basis representations of the subspaces $L$ contained in the SC derivative of a subdifferential mapping is more related to the concept of twice differentiability than the computation of the B-differential of the proximal mapping.

 \section{On the second-order qualification condition}

 Throughout this section we will suppose the following assumption.

\begin{assumption}\label{AssA1}
 The mapping $F:\R^n\to\R^m$ is continuously differentiable around the reference point $\xb\in\R^n$ and $g:\R^n\to\oR$ is an \lsc convex function such that  $\yb:=F(\xb)\in\dom g$. Finally, we are given a subgradient $\yba\in\partial g(\yb)$ and set $\xba:=\nabla F(\xb)^T\yba$.
 \end{assumption}

 The following theorem is fundamental for our analysis.

 \begin{theorem}\label{ThSOQC}
   The following conditions are equivalent:
   \begin{enumerate}
     \item[(i)]
     \begin{equation}
   \label{EqSOQC_SCD}\Big(\nabla F(\xb)^Tv^*=0,\ (0,v^*)\in L\ \Rightarrow\ v^*=0\Big)\quad\mbox{for all }L\in\Sp(\partial g)(\yb,\yba).
 \end{equation}
     \item[(ii)]
     \begin{equation}
   \label{EqSOQC_quad}\Big(\nabla F(\xb)^Tv^*=0,\ v^*\in \partial q(0)\ \Rightarrow\ v^*=0\Big)\quad\mbox{for all }q\in{\rm quad\,}g(\yb,\yba).
 \end{equation}
 \item[(iii)]
 \begin{equation}
 \label{EqSOQC_Coder} \nabla F(\xb)^Tv^*=0,\ v^*\in D^*(\partial g)(\yb,\yba)(0)\ \Rightarrow\ v^*=0.
 \end{equation}
 \item[(iv)]
 \begin{equation}
 \label{EqSOQC_StrictDer} \nabla F(\xb)^Tv^*=0,\ v^*\in D_*(\partial g)(\yb,\yba)(0)\ \Rightarrow\ v^*=0.
 \end{equation}
 \item[(v)]The mapping $\Psi:\R^n\times \R^m\tto\R^m$ given by
 \begin{equation}
   \Psi(x,y^*)=(\partial g)^{-1}(y^*)-F(x)
 \end{equation}
 is metrically regular around $\big((\xb,\yba),\bb)$.
 \item[(vi)]The multiplier mapping $M_{\xb}:\R^n\times\R^m\tto\R^m$ given by
 \[M_{\xb}(x^*,b)=\{y^*\in\partial g(\yb+b)\mv \nabla F(\xb)^Ty^*=x^*\}\]
 has a graphical localization around $\big((\xba,\bb),\yba\big)$ which is single-valued and Lipschitzian relative to its domain.
   \end{enumerate}
 \end{theorem}
 \begin{proof}The equivalence (i)$\Leftrightarrow$(ii) is an easy consequence of Proposition \ref{PropSCD_quad}. The implication (iii)$\Rightarrow$(i) follows from the second inclusion in \eqref{EqIncl1} together with the fact that $\Sp(\partial g)(\yb,\yba)=\Sp^*(\partial g)(\yb,\yba)$. We prove the implication (i)$\Rightarrow$(iii) by contraposition. Assume on the contrary to (iii) that there is some nonzero $v^*$ satisfying $\nabla F(\xb)^Tv^*=0$ and $(0,v^*)\in\gph D^*(\partial g)(\yb,\yb^*)$. By the second inclusion in \eqref{EqInclCoderiv} we can find some matrix $B\in\co\onabla\prox g(\yb+\yba)$ and some $p\in \R^m$ with $Bp=0$ and $v^*=(I-B)p=p$. Since the B-differential $\onabla\prox g(\yb+\yba)$ is compact, there exists a natural number $N$, positive reals $\alpha_i$ and matrices $B_i\in\onabla\prox g(\yb,\yba)$, $i=1,\ldots,N$ such that $\sum_{i=1}^N\alpha_i=1$ and $\sum_{i=1}^N\alpha_iB_i=B$. From $Bp=0$ we readily infer that
 \[\skalp{p,Bp}=\sum_{i=1}^N\alpha_i\skalp{p,B_ip}=0,\]
 implying $\skalp{p, B_ip}=0$, $i=1,\ldots,N$ by the positive semidefiniteness of the matrices $B_i$ and, together with their symmetry, $B_ip=0$, $i=1,\ldots,N$ follows. By \eqref{EqSCDconvex} we obtain that $(0,v^*)=(B_ip,(I-B_i)p)\in\rge(B_i,I-B_i)\in \Sp(\partial g)(\yb,\yba)$, $i=1,\ldots,N$,  a contradiction to \eqref{EqSOQC_SCD}. Hence the implication (i)$\Rightarrow$(iii) holds true and the equivalence (i)$\Leftrightarrow$(iii) is established.\\
 The equivalence between (i) and (iv) can be shown in the same way by using the first inclusion in \eqref{EqIncl1} and \eqref{EqInclStrictDeriv}.\\
 In order to show the equivalence between (iii) and (v), note that
 \[D^*\Psi(\xb,\yba)(v^*)=\{(-\nabla F(\xb)^Tv^*,u)\mv u\in D^*(\partial g)^{-1}(\yba,\yb)(v^*)\},\ v^*\in\R^m.\]
  Since $D^*(\partial g)^{-1}(\yba,\yb)(v^*)=\{u\mv -v^*\in D^*(\partial g)(\yb,\yba)(-u)\}$, the Mordukhovich criterium tells us that (v) holds if and only if (iii) is valid.\\
  Finally, in order to show the equivalence between (iv) and (vi), consider the mapping $\Gamma:\R^m\tto \R^n\times\R^m$ given by
  \[\Gamma(y^*):=\big(\nabla F(\xb)^Ty^*, (\partial g)^{-1}(y^*)-\yb\big),\ y^*\in\R^m.\]
By \cite[Exercise 10.43]{RoWe98} we obtain that
  \[\gph D_*\Gamma\big(\yba,(\xba,\bb)\big)= \big\{\big(v^*,(\nabla F(\xb)^Tv^*, v)\big)\mv (v^*,v)\in\gph D_*(\partial g)^{-1}(\yba,\yb)\big\}.\]
  Hence, by taking into account the equation $\gph D_*(\partial g)^{-1}(\yba,\yb)=\{(v^*,v)\mv (v,v^*)\in\gph D_*(\partial g)(\yb,\yba)\}$, condition (iv) is equivalent with  $\ker D_*\Gamma\big(\yba,(\xba,\bb)\big)=\{0\}$. Since $M_{\xb}=\Gamma^{-1}$, we conclude from Rockafellar's criterion in Theorem \ref{ThMordRoCrit}(ii) that statement (iv) is equivalent with  (vi). This completes the proof.
 \end{proof}
 \begin{definition}
 We will say that the {\em second-order qualification condition} (SOQC) holds for the composite function $g\circ F$ at $\xb$ for $\yba$  if any of the four equivalent conditions \eqref{EqSOQC_SCD}, \eqref{EqSOQC_quad}, \eqref{EqSOQC_Coder} or \eqref{EqSOQC_StrictDer} is fulfilled.
 \end{definition}
 The SOQC in terms of limiting coderivatives \eqref{EqSOQC_Coder} goes back to \cite{MoRo12}, where it was introduced in the form
 \[D^*(\partial g)(\yb,\yba)(0)\cap\ker \nabla F(\xb)^T=\{0\}.\]
 For the special case when $g$ is the indicator function of a closed convex set it was already employed in \cite{MoOut07}.
 It was shown in \cite{KhMoPh23} that  SOQC reduces to the linear independence constraint qualification fin case of the nonlinear programming problem.
 \begin{remark}\label{RemEquiv_SOQC_SCD}
   Given $(P,W)\in\M_{P,W}\partial g(\yb,\yba)$, an element $(0,v^*)\in\rge(P,W)$ has the representation $v^*=Wp$ with $Pp=0$. Hence $p=(I-P)p$ implying
   $v^*=W(I-P)p=(I-P)p\in \rge(I-P)=(\rge P)^\perp$ and we obtain the equivalence
   \begin{equation}
     \label{EqEquivSOQC_SCD}\eqref{EqSOQC_SCD}\Longleftrightarrow\Big(\ker \nabla F(\xb)^T\cap(\rge P)^\perp=\{0\}\ \forall (P,W)\in\M_{P,W}\partial g(\yb,\yba)\Big).
   \end{equation}
   Further, since $\rge P=\dom q_{P,W}$, we infer from \eqref{EqQuad_PW_Equiv} that
   \begin{equation}
     \label{EqEquivSOQC_quad}\eqref{EqSOQC_quad}\Longleftrightarrow\Big(\ker \nabla F(\xb)^T\cap(\dom q)^\perp=\{0\}\ \forall q\in{\rm quad\,}g(\yb,\yba)\Big).
   \end{equation}
 \end{remark}
 The SOQC  is persistent with respect to small perturbations in $\xb,\yba$ and $F$. We confine ourselves with the following result.
 \begin{proposition}\label{PropSOQC_stable}
   Assume that the SOQC  holds at $\xb$ for $\yba$. Then there is a neighborhood ${\cal V}\subset\R^n\times\R^m\times\R^m$ of $(\xb,\yba,\bb)$ such that for all $(\tilde x,\tilde y^*,b)\in {\cal V}$ satisfying $\tilde y^*\in\partial g(F(\tilde x)+b)$ one has
   \[\Big(\nabla F(\tilde x)^Tv^*=0,\ (0,v^*)\in L\ \Rightarrow\ v^*=0\Big)\quad\mbox{for all }L\in\Sp(\partial g)(F(\tilde x)+b,\tilde y^*),\]
   i.e., the SOQC holds for the perturbed mapping $g\circ F_b$ at $\tilde x$ for $\tilde y^*$ with $F_b:\R^n\to\R^m$ given by
   \begin{equation}\label{EqF_b}F_b(x):=F(x)+b\end{equation}.
 \end{proposition}
 \begin{proof}Assume on the contrary that there are  sequences $(x_k,y_k^*,b_k)\to(\xb,\yba,\bb)$, $L_k\in\Z_m$ and  $0\not=v_k^*\in \R^m$ satisfying $y_k^*\in\partial g(F(x_k)+b_k)$, $L_k\in \Sp(\partial g)(F(x_k)+b_k,y_k^*)$, $(0,v_k^*)\in L_k$ and  $\nabla F(x_k)^Tv_k^*=0$ for all $k$. We may assume that $\norm{v_k^*}=1$, and since the metric space $\Z_m$ is compact, cf. \cite[Lemma 3.1]{GfrOut22}, after possibly passing to a subsequence we may suppose that $L_k$ converges to some subspace $L\in\Z_m$ and that $v_k^*$ converges to some $v^*\in\R^m$ with $\norm{v^*}=1$. Since the orthogonal projections $P_{L_k}$ onto $L_k$ converge to the orthogonal projection $P_L$ onto $L$, we conclude that
 \[(0,v^*)=\lim_{k\to\infty}(0,v_k^*)=\lim_{k\to\infty}P_{L_k}(0,v_k^*)=P_L(0,v^*)\]
 showing $(0,v^*)\in L$. Further, since $\nabla F(\xb)^Tv^*=\lim_{k\to\infty}\nabla F(x_k)^Tv_k^*=0$ and $L\in\Sp(\partial g)(\yb,\yba)$ by \cite[Lemma 3.13]{GfrOut22}, we obtain a contradiction to the condition \eqref{EqSOQC_SCD} defining SOQC.
 \end{proof}
 \begin{lemma}\label{LemBCQ}
   Assume that the SOQC holds at $\xb$ for $\yba$. Then the {\em basic constraint qualification}
   \begin{equation}\label{EqBQC}
     \nabla F(\xb)^Tv^*=0,\ v^*\in N_{\dom g}(F(\xb))\ \Rightarrow v^*=0
   \end{equation}
   also holds, implying that there is an open neighborhood $\U$ of $(\xb,\bb)$ such that for every $(x,b)\in \U$ with $F_b(x)\in\dom g$ one has
   \begin{equation}\label{EqBQC1}
     \nabla F_b(x)^Tv^*=0,\ v^*\in N_{\dom g}(F_b(x))\ \Rightarrow v^*=0
   \end{equation}
   In particular, for every such $(x,b)$ one has $\partial(g\circ F_b)(x)=\nabla F_b(x)^T\partial g(F_b(x))$ and,  if $F$ is twice continuously differentiable around  $\xb$, the composite function $g\circ F_b$ is prox-regular and subdifferentially continuous at $x$.
 \end{lemma}
 \begin{proof}
   Assume on the contrary to \eqref{EqBQC} that there is some $0\not=v^*\in N_{\dom g}(\yb)$ satisfying $\nabla F(\xb)^Tv^*=0$. Then for every $\alpha\geq 0$ we have $\yba+\alpha v^*\in\partial g(\yb)$ implying $M_{\xb}(\xba,\bb)\supset \{\yba+\alpha v^*\mv\alpha\geq 0\}$. This contradicts statement (vi) of Theorem \ref{ThSOQC} and hence SOQC is also violated. Thus, the basic qualification condition \eqref{EqBQC} holds true. Now \eqref{EqBQC1} is an easy consequence of the relationships $\Limsup_{y\longsetto{\dom g}\yb}N_{\dom g}(y)\subset N_{\dom g}(\yb)$ and  $\Limsup_{x\to\xb}\ker \nabla F(x)^T\subset\ker \nabla F(\xb)^T$. Condition \eqref{EqBQC1} ensures that the composite function $g\circ F_b$ is
    amenable at $x$ in the sense of \cite[Definition 10.23]{RoWe98} and the assertions about the chain rule, prox-regularity and subdifferential continuity follow from \cite[Exercise 10.25, Theorem 10.6, Proposition 13.32]{RoWe98}.
 \end{proof}
The next result states uniqueness of multipliers together with some Lipschitz property. To this aim we define the multiplier mapping $M:\R^n\times\R^n\times\R^m\tto\R^m$ by
\begin{equation}
  \label{EqMultMap}M(x,x^*,b):=\{y^*\in \partial g(F(x)+b)\mv \nabla F(x)^Ty^*=x^*\}.
\end{equation}
\begin{theorem}\label{ThLipMult}
  Assume that SOQC holds at $\xb$ for $\yba$ and that $\nabla F$ is Lipschitz near $\xb$. Then there exists a neighborhood $\cal W$ of $(\xb,\xba,\bb)$ together with a constant $\kappa\geq 0$ such that for every $(x_i,x_i^*,b_i)\in\dom M\cap {\cal W}$ and every $y_i^*\in M(x_i,x_i^*,b_i)$, $i=1,2$, there holds
  \begin{equation}
    \norm{y_1^*-y_2^*}\leq\kappa\norm{(x_1,x_1^*,b_1)-(x_2,x_2^*,b_2)}.
  \end{equation}
  In particular, for every $(x,x^*,b)\in\dom M\cap {\cal W}$ the set $M(x,x^*,b)$ is a singleton.
\end{theorem}
\begin{proof}
  We claim  that for every sequence $(x_k,x_k^*,b_k,y_k^*)\in\gph M$ satisfying $(x_k,x_k^*,b_k)\to (\xb,\xba,\bb)$ there holds $y_k^*\to\yba$ as $k\to\infty$. Assume on the contrary that we can find a sequence $(x_k,x_k^*,b_k,y_k^*)\in\gph M$ satisfying $(x_k,x_k^*,b_k)\to (\xb,\xba,\bb)$ and $\norm{y_k^*-\yba}>\epsilon>0$ $\forall k$. Set $y_k:=F(x_k)+b_k\in(\partial g)^{-1}(y_k^*)$ $\forall k$. If the sequence $y_k^*$ possesses a subsequence converging  to some $\tilde y^*$, then $\tilde y^*\in\partial g(\yb)$ and  $\nabla F(\xb)^T\tilde y^*=\xba$ implying that $M_{\xb}(\xba,\bb)\supset\{\alpha \yba+(1-\alpha)\tilde y^*\mv \alpha\in[0,1]\}$ contradicting statement Theorem \ref{ThSOQC}(vi) because of $\norm{\tilde y^*-\yba}\geq\epsilon$. Thus, SOQC is also violated.\\
  If the sequence $y_k^*$ does not have any convergent subsequence, there holds $\norm{y_k^*}\to\infty$ and, by possibly passing to a subsequence, we may assume that the bounded sequence $v_k^*:=y_k^*/\norm{y_k^*}$ converges to some $v^*\in \partial^\infty g(\yb)=N_{\dom g}(\yb)$ with $\norm{v^*}=1$. Further,
  \[\nabla F(\xb)^Tv^*=\lim_{k\to\infty}\nabla F(x_k)^Tv_k^*=\lim_{k\to\infty}\frac{x_k^*}{\norm{y_k^*}}=0\]
  contradicting \eqref{EqBQC} and consequently also SOQC by Lemma \ref{LemBCQ}. Hence our claim holds true.

  We now prove the assertion of the proposition by contraposition. Assume on the contrary that we can find two sequences $(x_{i,k},x_{i,k}^*,b_{i,k}, y_{i,k}^*)\in\gph M$, $i=1,2$,  satisfying $\lim_{k\to\infty}(x_{i,k},x_{i,k}^*,b_{i,k})=(\xb,\xba,\bb)$, $i=1,2$, such that
  \begin{equation}\label{EqAuxLip}\norm{y_{1,k}^*-y_{2,k}^*}> k\norm{(x_{1,k},x_{1,k}^*,b_{1,k})-(x_{2,k},x_{2,k}^*,b_{2,k})}\ \forall k.\end{equation}
  Then we have $\norm{y_{1,k}^*-y_{2,k}^*}>0$ $\forall k$ and
   \[\lim_{k\to\infty}\frac{\norm{(x_{1,k},x_{1,k}^*,b_{1,k})-(x_{2,k},x_{2,k}^*,b_{2,k})}}{\norm{y_{1,k}^*-y_{2,k}^*}}=0.\]
  By possibly passing to a subsequence, we may assume that the sequence
  \[v_k^*:=\frac{y_{1,k}^*-y_{2,k}^*}{\norm{y_{1,k}^*-y_{2,k}^*}}\]
  converges to some $v^*$ with $\norm{v^*}=1$. Since $F$ is continuously differentiable, we have $F(x_{1,k})-F(x_{2,k})=\nabla F(\xb)(x_{1,k}-x_{2,k})+\oo(\norm{x_{1,k}-x_{2,k}})$. Setting $y_{i,k}:=F(x_{i,k})+b_{i,k}\in(\partial g)^{-1}(y_{i,k}^*)$ we obtain that
  \[\lim_{k\to\infty}\frac{y_{1,k}-y_{2,k}}{\norm{y_{1,k}^*-y_{2,k}^*}}= \lim_{k\to\infty}\frac{\nabla F(\xb)(x_{1,k}-x_{2,k})+\oo(\norm{x_{1,k}-x_{2,k}}) + b_{1,k}-b_{2,k}}{\norm{y_{1,k}^*-y_{2,k}^*}}=0.\]
  Since by our just proven claim we have $\lim_{k\to\infty}y_{i,k}^*=\yba$, $i=1,2$, the inclusion $(0,v^*)\in \gph D_*(\partial g)(\yb,\yba)$ follows.
  From the relationship
  \[x^*_{1,k}-x^*_{2,k}=\nabla F(x_{1,k})^T(y_{1,k}^*-y_{2,k}^*)+(\nabla F(x_{1,k})^T-\nabla F(x_{2,k})^T)y_{2,k}^*\]
  together with the assumed Lipschitz continuity of $\nabla F$ we infer that
  \[\nabla F(\xb)^Tv^*=\lim_{k\to\infty}\frac{\nabla F(x_{1,k})^T(y_{1,k}^*-y_{2,k}^*)}{\norm{y_{1,k}^*-y_{2,k}^*}}=\lim_{k\to\infty}\frac{x^*_{1,k}-x^*_{2,k}-(\nabla F(x_{1,k})^T-\nabla F(x_{2,k})^T)y_{2,k}^*}{\norm{y_{1,k}^*-y_{2,k}^*}}=0.\]
  Thus $v^*$ contradicts \eqref{EqSOQC_StrictDer} and consequently SOQC is also violated. The proof is complete.
\end{proof}
 We now state the following exact  chain rules  for the graphical derivative of $\partial(g\circ F_b)$, where $F_b:\R^n\to\R^m$ is given by \eqref{EqF_b}.

\begin{proposition}\label{PropChainRuleGraphDer}
  Assume that $F$ is twice continuously differentiable around $\xb$ and assume that SOQC holds at $\xb$ for $\yba$. Then there is a neighborhood ${\cal W}$ of $(\xb,\xba,\bb)$ such that for every $(\tilde x,\tilde x^*,b)\in {\cal W}$ verifying $\tilde x^*\in\partial (g\circ F_b)(\tilde x)$ there holds
  \begin{equation}\label{EqChainTanCone}T_{\gph \partial(g\circ F_b)}(\tilde x,\tilde x^*) = \{(u,\nabla^2\skalp{\tilde y^*,F_b}(\tilde x)u+\nabla F_b(\tilde x)^Tv^*)\mv (\nabla F_b(\tilde x)u,v^*)\in T_{\gph\partial g}(F_b(\tilde x),\tilde y^*)\}\end{equation}
  with $\{\tilde y^*\}=M(\tilde x,\tilde x^*,b)$, and therefore the second-order chain rule
  \[D(\partial (g\circ F_b))(\tilde x,\tilde x^*)(u)=\nabla^2\skalp{\tilde y^*,F_b}(\tilde x)u+\nabla F_b(\tilde x)^TD(\partial g)(F_b(\tilde x), \tilde y^*)(\nabla F_b(\tilde x)u),\ u\in\R^n\]
  holds true.
\end{proposition}
\begin{proof}Let us choose an open neighborhood ${\cal W}$ of $(\xb,\xba,\bb)$  meeting the requirements of Theorem \ref{ThLipMult} and such that for all $(x,x^*,b)\in\W$ the pair $(x,b)$ belongs to the neighborhood $\U$ from Lemma \ref{LemBCQ}.
Consider $(\tilde x,\tilde x^*,b)\in {\cal W}$ verifying $\tilde x^*\in\partial (g\circ F_b)(\tilde x)=\nabla F_b(x)^T\partial g(F_b(x))$. By taking into account the equation $\dom M =\{(x,x^*,b)\mv x^*\in \nabla F_b(x)^T\partial g(F_b(x))\}$, we conclude from Theorem \ref{ThLipMult} that $M(\tilde x,\tilde x^*,b)=\{\tilde y^*\}$ is a singleton and $\norm{\tilde y^*-\yba}\leq \kappa \norm{(\tilde x,\tilde x^*,b)-(\xb,\xba,\bb)}$. Hence, by possibly shrinking $\W$ and taking into account Proposition \ref{PropSOQC_stable}, we may also assume that SOQC holds for the function $g\circ F_b$ at $\tilde x$ for $\tilde y^*$.

Now consider $(u,u^*)\in T_{\gph \partial(g\circ F_b)}(\tilde x,\tilde x^*)$ together with sequences $t_k\downarrow 0$ and $(u_k,u_k^*)\to(u,u^*)$ satisfying $(\tilde x+t_k u_k,\tilde x^*+t_k u_k^*)\in \gph \partial(g\circ F_b)$ for all $k$. For all $k$ sufficiently large we have $(\tilde x+t_ku_k,\tilde x^*+t_ku_k^*,b)\in\W$ and we infer from Theorem \ref{ThLipMult} that  the set $M(\tilde x+t_ku_k,\tilde x^*+t_ku_k^*,b)=\{y_k^*\}$ is a singleton  and $\norm{y_k^*-\tilde y^*}\leq \kappa t_k\norm{u_k,u_k^*,0}$. After possibly passing to a subsequence, we can assume that $v_k^*:=(y_k^*-\tilde y^*)/t_k$ converges to some $v^*$.
From $y_k^*=\tilde y+t_kv_k\in\partial g(F_b(\tilde x+t_ku_k))$ together with $F_b(\tilde x+t_ku_k)=F_b(\tilde x)+t_k\nabla F_b(\tilde x)u+\oo(t_k)$ we conclude that $(\nabla F_b(\tilde x)u,v^*)\in T_{\gph \partial q}(F_b(\tilde x),\tilde y^*)$. Since
\[x_k^*=\tilde x^*+t_ku_k^*=\nabla F_b(\tilde x+t_ku_k)^Ty_k^*=\nabla F_b(\tilde x)^T\tilde y^*+t_k(\nabla^2\skalp{\tilde y^*,F_b}(\tilde x)u_k+\nabla F_b(\tilde x)^Tv_k^*)+\oo(t_k),\]
it follows that $u^*= \nabla^2\skalp{\tilde y^*,F_b}(\tilde x)u+\nabla F_b(\tilde x)^Tv^*$ and hence $(u,u^*)$ belongs to the set on the right hand side of \eqref{EqChainTanCone}.

Next consider a pair $(u,v^*)$ satisfying $(\nabla F_b(\tilde x)u, v^*)\in T_{\gph \partial g}(F_b(\tilde x), \tilde y^*)$. Then there are sequences $t_k\downarrow 0$ and $(\triangle v_k,\triangle v_k^*)\to(0,0)$ such that $(F_b(\tilde x)+t_k(\nabla F_b(\tilde x)u+\triangle v_k), \tilde y^*+t_k(v^*+\triangle v_k^*))\in\gph \partial g$ for all $k$. Hence
\[(\partial g)^{-1}(\tilde y^*+t_k(v^*+\triangle v_k^*))-F_b(\tilde x+t_ku)\ni F_b(\tilde x)+t_k(\nabla F_b(\tilde x)u+\triangle v_k)-F_b(\tilde x+t_ku)=\oo(t_k)\]
and, since the mapping $(x,y^*)\tto (\partial g)^{-1}(y^*)-F_b(x)$ is metrically regular around $((\tilde x,\tilde y^*),0)$ by Proposition \ref{PropSOQC_stable} and Theorem \ref{ThSOQC}(v), we can find for every $k$ sufficiently large some $(u_k,v_k^*)$ such that $\norm{(u_k,v_k^*)-(u,v^*+\triangle v_k^*)}=\oo(t_k)/t_k$ and $0\in (\partial g)^{-1}(\tilde y^*+t_kv_k)-F_b(\xb+t_ku_k)$. Hence
\begin{align*}\partial (g\circ F_b)(\tilde x+t_ku_k)\ni \nabla F(\tilde x+t_ku_k)^T(\tilde y^*+t_kv_k^*)=\nabla F(\tilde x)^T\tilde y^*+t_k\big(\nabla^2\skalp{\tilde y^*,F}(\tilde x)u+\nabla F(\tilde x)^Tv^*\big)+\oo(t_k)\end{align*}
showing that $(u, \nabla^2\skalp{\tilde y^*,F}(\tilde x)u+\nabla F(\tilde x)^Tv^*)\in T_{\gph\partial(g\circ F_b)}(\tilde x,\tilde x^*)$. Thus, equality holds in \eqref{EqChainTanCone} and the proof is complete.
\end{proof}
Finally we also provide an exact chain rule for the SC derivative of $\partial(g\circ F)$.
\begin{proposition}\label{PropChainRuleSCD}
  Assume that $F$ is twice continuously differentiable around $\xb$ and that either $\nabla F(\xb)$ has full row rank $m$ or SOQC holds at $\xb$ for $\yba$ and $\Sp(\partial g)(F(\xb),\yba)=\{\bar S\}$ is a singleton. Then
  \begin{equation}
    \label{EqChainRuleSCD}
    \Sp(\partial(g\circ F))(\xb,\xba)=\Big\{\big\{\big(u,\nabla^2\skalp{\yba,F}(\xb)u+\nabla F(\xb)^Tv^*\big)\mv (\nabla F(\xb)u,v^*)\in S\big\}\mv S\in \Sp(\partial g)(\yb,\yba)\Big\}.
  \end{equation}
\end{proposition}
\begin{proof} We begin with the case when $\nabla F(\xb)$ has full row rank $m$. Consider $(x,x^*)\in \partial(g\circ F)$ such that $(x,x^*,0)$ belongs to the neighborhood $\W$ of Proposition \ref{PropChainRuleGraphDer} and $\nabla F(x)$ has full row rank $m$. Let $y^*$ be the unique element in $M(x,x^*,0)$.  Any $u$ satisfying $\nabla F(x)u=v$ can be written in the form $u=\nabla F(x)^\dag v+p$ with $p\in\ker \nabla F(x)$ and, by identifying $\R^n\times\R^n$ with $\R^{2n}$, we obtain from \eqref{EqChainTanCone} together with the surjectivity of $\nabla F(x)$, that
\begin{align}\label{EqAux1ChRule}T_{\gph\partial(g\circ F)}(x,x^*)=\Big\{ A(x,y^*)\myvec{v\\v^*}+B(x,y^*)p\mv (v,v^*)\in T_{\gph \partial g}(F(x),y^*),\ p\in \ker\nabla F(x)\Big\}\end{align}
with
\[A(x,y^*):=\begin{pmatrix}\nabla F(x)^\dag &0\\\nabla^2\skalp{y^*,F}(x)\nabla F(x)^\dag& \nabla F(x)^T\end{pmatrix},\ B(x,y^*)=\begin{pmatrix}I_n\\\nabla^2\skalp{y^*,F}(x)\end{pmatrix}.\]
We see that $T_{\gph\partial(g\circ F)}(x,x^*)$ is a subspace, whenever $(F(x),y^*)\in\OO_{\partial g}$, and its dimension must be $n$, because $g\circ F$ is prox-regular and subdifferentially continuous at $x$. Likewise,
\[T_{\gph \partial g}(F(x),y^*)=\Big\{C(x,y^*)\myvec{u\\u^*}\mv (u,u^*)\in T_{\gph\partial(g\circ F)}(x,x^*) \Big\}\]
with
\[C(x,y^*)= \begin{pmatrix}
  \nabla F(x)&0\\-{\nabla F(x)^T}^\dag\nabla^2\skalp{\tilde y^*,F}(x)&{\nabla F(x)^T}^\dag\end{pmatrix}\]
  and therefore $(F(x),y^*)\in\OO_{\partial g}$ whenever $(x,x^*)\in\OO_{\partial(g\circ F)}$. Hence $(x,x^*)\in \OO_{\partial(g\circ F)}$ if and only if $(F(x),y^*)\in\OO_{\partial g}$.

  Next choose $L\in\Sp(\partial(g\circ F))(\xb,\xba)$ together with a sequence $(x_k,x_k^*)\longsetto{\OO_{\partial(g\circ F)}}(\xb,\xba)$ such that
  $L_k:=T_{\gph\partial(g\circ F)}(x_k,x_k^*)\to L$ in $\Z_n$. For all $k$ sufficiently large let $S_k:=T_{\gph \partial g}(F(x_k),y_k^*)$,   $\{y_k^*\}=M(x_k,x_k^*,0)$, denote the corresponding subspaces in $\Z_m$. Since $\Z_m$ is compact, by possibly passing to a subsequence we may assume that $S_k$ converges to some $S\in\Sp(\partial g)(\yb,\yba)$ and we can define the subspace
  \[\hat L:=\big\{ A(\xb,\yba)\myvec{v\\v^*}+B(\xb,\yba)p\mv (v,v^*)\in S,\ p\in \ker\nabla F(\xb)\big\}.\]
 In view of \eqref{EqAux1ChRule} we have $\hat L\subset L$ and in order to show equality we must verify that $\dim \hat L=\dim L(=n)$. Consider the linear mapping $H:S\times \ker \nabla F(\xb)\to\R^{2n}$ given by
 \[H\big((v,v^*),p \big):= A(\xb,\yba)\myvec{v\\v^*}+B(\xb,\yba)p=\myvec{\nabla F(\xb)^\dag v + p\\\nabla^2\skalp{\yba,F}(\xb)(\nabla F(\xb)^\dag v +p)+ \nabla F(\xb)^T v^*}\]
 and  $\big((v,v^*),p \big)\in S\times \ker F(\xb)$ satisfying $H\big((v,v^*),p \big)=0$.
 Since $\rge \nabla F(\xb)^\dag =(\ker \nabla F(\xb))^\perp$ and $\nabla F(x)^\dag$ has full column rank $m$, we obtain that $v=0$ and $p=0$. Hence $\nabla F(\xb)^Tv^*=0$ and this is only possible when $v^*=0$. We infer that $H$ is injective and, together with $\hat L=H(S\times\ker F(\xb))$ and $\dim (S\times\ker F(\xb))=m+(n-m)=n$, we infer that $\dim \hat L=n$ and consequently $\hat L=L$.
  This verifies the inclusion ``$\subset$'' in \eqref{EqChainRuleSCD}.

 Now consider a subspace $S\in\Sp(\partial g)(\yb,\yba)$ together with a sequence $(y_k,y_k^*)\longsetto{\OO_{\partial g}}(\yb,\yba)$ such that $S_k:= T_{\gph\partial g}(y_k,y_k^*)$ converges to $S$. The full rank assumption on $\nabla F(\xb)$ ensures that for all $k$ sufficiently large we can find some $x_k$ such that $F(x_k)=y_k$ so that the sequence $x_k$ converges to $\xb$. By setting $x_k^*=\nabla F(x_k)y_k^*$, we conclude that $(x_k,x_k^*)\in\OO_{\partial(g\circ F)}$ and, after possibly passing to a subsequence, we may assume that the subspaces
 \[L_k:=T_{\gph\partial(g\circ F)}(x_k,x_k^*)=\big\{A(x_k,y_k^*)\myvec{v\\v^*}+B(x_k,y_k^*)p\mv (v,v^*)\in S_k,\ p\in \ker\nabla F(x_k)\big\}\]
 converge to some $L\in\Sp(\partial(g\circ F))(\xb,\xba)$. The same arguments as before yield that $L=H(S\times\ker F(\xb))$
 and  the inclusion ``$\supset$'' in \eqref{EqChainRuleSCD} also follows. This completes the proof of the case when $\nabla F(\xb)$ has full row rank $m$.

 We proceed with the second case when $\Sp(\partial g)(\yb,\yba)=\{\bar S\}$ is a singleton. By \cite[Corollary 3.3]{Gfr25c} we infer that $\gph D^*(\partial g)(\yb,\yba)=\bar S$ and hence  the inclusion
 \[\gph D^*(\partial(g\circ F))(\xb,\xba)\subset \{(u, \nabla^2\skalp{\yba,F}(\xb)u+\nabla F(\xb)^Tv^*)\mv (\nabla F(\xb)u,v^*)\in \bar S\}=:\bar L\]
 follows from \cite[Theorem 3.3]{MoRo12}. Clearly, $\bar L$ is a subspace and we will now show that $\dim\bar L\leq n$. Taking $\tilde S=\bar S\cap(\rge \nabla F(\xb)\times \R^m)$, we conclude that
 \[\bar L=\{(\nabla F(\xb)^\dag v+p, \skalp{\yba,F}(\xb)(\nabla F(\xb)^\dag v+p)+\nabla F(\xb)^Tv^*)\mv (v,v^*)\in\tilde S, p\in\ker \nabla F(\xb)\}.\]
 Assuming that $\dim\bar L>n$, we can find $n+1$ elements $\big((v_i,v_i^*),p_i\big)\in \tilde S$ such that
 \[(u_i,u_i^*):=(\nabla F(\xb)^\dag v_i+p_i, \skalp{\yba,F}(\xb)(\nabla F(\xb)^\dag v_i+p_i)+\nabla F(\xb)^Tv_i^*),\ i=1,\ldots,n+1\]
 are linearly independent. Since $\dim\rge\nabla F(\xb)^\dag + \dim\ker\nabla F(\xb)=n$, we can find scalars $\alpha_1,\ldots,\alpha_{n+1}$ with $\sum_{i=1}^{n+1}\alpha_iu_i=0$ and $\bar u^*:=\sum_{i=1}^{n+1}\alpha_iu_i^*\not=0$. Setting $\big((\bar v,\bar v^*),\bar p\big):=\sum_{i=1}^{n+1}\alpha_i\big((v_i,v_i^*),p_i\big)$, it follows that $\bar u^*=\nabla F(\xb)^T\bar v^*$, $(\bar v,\bar v^*)\in\tilde S$, $\bar p\in\ker\nabla F(\xb)$ and $\nabla F(\xb)^\dag\bar v+\bar p=0$. Since $\rge \nabla F(\xb)^\dag=(\ker\nabla F(\xb))^\perp$, we conclude that $\bar p=0$ and $\nabla F(\xb)^\dag\bar v=0$, the latter implying $\bar v=0$ because of $\bar v\in\rge\nabla F(\xb)$. Thus $(\bar v,\bar v^*)=(0,\bar v^*)\in S$ and $\nabla F(\xb)^T\bar v^*=\bar u^*\not=0$, contradicting condition \eqref{EqSOQC_SCD} defining SOQC. Hence $\dim\bar L\leq n$ and
 since $\Sp(\partial (g\circ F))(\xb,\xba)$ is nonempty, we conclude from \eqref{EqIncl1} that $\dim\bar L=n$ and $\Sp(\partial (g\circ F))(\xb,\xba)=\{\bar L\}$. This completes the proof.
\end{proof}

\section{Characterizations of the Aubin property of the KKT mapping}

Throughout this section we will suppose that  the following assumption is fulfilled.
 \begin{assumption}\label{Ass2}
 In addition to Assumption \ref{AssA1}, the mappings $F:\R^n\to\R^m$ and $f:\R^n\to\R$ are twice continuously differentiable around the reference point $\xb$ and the first-order necessary optimality condition
\[\nabla_x\lag(\xb,\yba)=0\]
is fulfilled.
\end{assumption}

We start our analysis with the Mordukhovich criterion for the Aubin property of $\SKKT$. Note that the KKT conditions \eqref{EqKKTpert} can be written in the form
\begin{equation}\label{EqPsi}
(a^*,b)\in G(x,y^*):=\big(\nabla_x\lag(x,y^*), -F(x)\big)+\{0\}\times(\partial g)^{-1}(y^*)
\end{equation}
and therefore $\SKKT=G^{-1}$.
By taking into account \cite[Exercise 10.43]{RoWe98} and the equations
 \begin{align*}D^*(\partial g)^{-1}(\yba,\yb)(v^*)&=\{v\mv -v^*\in D^*(\partial g)(\yb,\yba)(-v)\},
 \end{align*}
we obtain that
\begin{align}\label{EqCoderivG}D^*G\big((\xb,\yba),(\aba,\bb)\big)(u,v^*)=\big\{(\nabla^2_{xx}\lag(\xb,\yba)u-\nabla F(\xb)^Tv^*, \nabla F(\xb)u+v)\mv -v^*\in D^*(\partial g)(\yb,\yba)(-v)\big\}.
\if{\\D_*G\big((\xb,\yba),(\aba,\bb)\big)(u,v^*)=\big\{(\nabla^2_{xx}\lag(\xb,\yba)u+\nabla F(\xb)^Tv^*, -\nabla F(\xb)u+w)\mv v^*\in D_*(\partial g)(\yb,\yba)(w)\big\}.}\fi
\end{align}
\begin{lemma}
\label{LemMord}$\SKKT$ has the Aubin property around $\big((\aba,\bb),(\xb,\yba)\big)$ if and only if
\begin{equation}\label{EqMord}\nabla^2_{xx}\lag(\xb,\yba)u+\nabla F(\xb)^Tv^*=0,\ v^*\in D^*(\partial g)(\yb,\yba)(\nabla F(\xb)u)\ \Rightarrow\ (u,v^*)=(0,0).
\end{equation}
\end{lemma}
\begin{proof}
Since $\gph G$ is locally closed around $\big((\xb,\yba),(\aba,\bb)\big)$, $\SKKT$ has the Aubin property around $\big((\aba,\bb),(\xb,\yba))$ if and only if $G$ is metrically regular around $\big((\xb,\yba),(\aba,\bb)\big)$, which in turn is equivalent to the Mordukhovich criterion
$(0,0)\in D^*G\big((\xb,\yba),(\aba,\bb)\big)(u,v^*)\ \Rightarrow\ (u,v^*)=(0,0)$. By using \eqref{EqCoderivG}, it is easy to see that the Mordukhovich criterion, after replacing $v^*$ by $-v^*$, is  equivalent to \eqref{EqMord}.
\end{proof}

The following statement is an immediate consequence of the preceding lemma.
\begin{proposition}\label{PropAubinSOQC}If $\SKKT$ has the Aubin property around $\big((\aba,\bb),(\xb,\yba)\big)$ then SOQC holds at $\xb$ for $\yba$.
\end{proposition}
\begin{proof}
  Follows from the observation that \eqref{EqMord} with $u=0$ amounts to condition \eqref{EqSOQC_Coder} defining SOQC.
\end{proof}
Now  consider the mapping $\Phi:\R^n\times\R^m\times\R^n\times\R^m\to \R^n\times\R^m\times\R^n\times\R^m$ given by
$\Phi(x,y^*,u^*,v)=(x, y^*+v+F(x),u^*,v)$. Then $\Phi$ is one-to-one, its Jacobian is everywhere invertible and
\[\Phi(\gph G)= \{(x, y+y^*, \nabla_x\lag(x,y^*),-F(x)+y)\mv x\in\R^n,\ (y,y^*)\in\gph \partial g\}=\gph \tilde G\]
with
\[\tilde G(x,z):=(\nabla f(x)+\nabla F(x)^T(z-\prox g(z)), -F(x)+\prox g(z)).\]
Clearly, $\tilde G$ is Lipschitz continuous near $(\xb,F(\xb)+\yba)$ and therefore $G$ is graphically Lipschitzian of dimension $n+m$ around $\big((\xb,\yba),(\aba,\bb)\big)$.

We are now in the position to state the first main result of this section.
\begin{theorem}
  Assume that the SC derivative $\Sp(\partial g)(F(\xb),\yba)=\{\bar L\}$ is a singleton. Then the following statements are equivalent.
  \begin{enumerate}
  \item[(i)]
  \begin{equation}\label{EqSCDRegular}
    \nabla^2_{xx}\lag(\xb,\yba)u+\nabla F(\xb)^Tv^*=0,\ (\nabla F(\xb)u,v^*)\in\bar L\ \Rightarrow\ (u,v^*)=(0,0).
  \end{equation}
    \item[(ii)] $\SKKT$ has the Aubin property around $\big((\aba,\bb),(\xb,\yba)\big)$.
    \item[(iii)] $\SKKT$ has a graphical localization $T$ around $\big((\aba,\bb),(\xb,\yba)\big)$ which is single-valued and Lipschitz continuous.
  \end{enumerate}
  Each of these statements entail that the single-valued graphical localization $T$ of $\SKKT$ is strictly differentiable at $(\aba,\bb)$
  and $\nabla T(\aba,\bb)(u^*,v)$ equals to the unique solution $(u,v^*)$ of the system
  \[\nabla^2_{xx}\lag(\xb,\yba)u+\nabla F(\xb)^Tv^*=u^*, (\nabla F(\xb)u+v,v^*)\in\bar L.\]
\end{theorem}
\begin{proof}
  Since $\Sp(\partial g)(\yb,\yba)=\{\bar L\}(=\Sp^*(\partial g)(\yb,\yba))$ and $\partial g$ is graphically Lipschitzian, we conclude from \cite[Corollary 3.3]{Gfr25c} that $\partial g$ is strictly proto-differentiable at $(\yb,\yba)$ and $\gph D^*(\partial g)(\yb,\yba)=\bar L.$
  Therefore
  \begin{equation}\label{EqAuxCoderivG}\gph D^*G\big((\xb,\yba),(\aba,\bb)\big)=\Big\{\big((u,v^*), (\nabla^2_{xx}\lag(\xb,\yba)u-\nabla F(\xb)^Tv^*, \nabla F(\xb)u+v)\big)\mv -(v,v^*)\in \bar L\Big\}\end{equation}
  is a subspace and, since $G$ is graphically Lipschitzian of dimension $n+m$, we can use \cite[Corollary 3.3]{Gfr25c} once more to obtain the strict proto-differentiability of $G$ at $\big((\xb,\yba),(\aba,\bb)\big)$. Thus  we can infer from \cite[Theorem 3.6]{Gfr25c} the equivalence of the following three statements:
  \begin{enumerate}
    \item[(a)]There is a $(n+m)\times(n+m)$ matrix $C$ such that
    \[\gph D^*G\big((\xb,\yba),(\aba,\bb)\big)=\rge(C^T,I_{n+m}).\]
    \item[(b)]$G$ is metrically regular around $\big((\xb,\yba),(\aba,\bb)\big)$.
    \item[(c)]$G$ is strongly metrically regular around $\big((\xb,\yba),(\aba,\bb)\big)$.
  \end{enumerate}
  Clearly, (ii)$\Leftrightarrow$(b) and (iii)$\Leftrightarrow$(c). By \cite[Proposition 4.2]{GfrOut22}, statement (a) is equivalent to the condition
  \[\nabla^2_{xx}\lag(\xb,\yba)u-\nabla F(\xb)^Tv^*=0,\ \nabla F(\xb)u+v=0,\ -(v,v^*)\in \bar L\ \Rightarrow\ (u,v^*)=(0,0),\]
  which is obviously equivalent to \eqref{EqSCDRegular}. Thus, (i)$\Leftrightarrow$(a) also holds and the equivalence of the three statements (i),(ii),(iii) is a consequence of the equivalence of (a),(b) and (c).\\
  The assertion about strict differentiability of the single-valued graphical localization $T$ of $\SKKT$ at $(\aba,\bb)$ follows from \cite[Theorem 3.6]{Gfr25c} and the formula for $\nabla T(\aba,\bb)(u^*v)=D\SKKT\big((\aba,\bb),(\xb,\yba)\big)(u^*,v)$ is a consequence of $D\SKKT\big((\aba,\bb),(\xb,\yba)\big)=DG\big((\xb,\yba),(\aba,\bb)\big)^{-1}$ together with
  \begin{align*}\gph DG\big((\xb,\yba),(\aba,\bb)\big)&=\gph D_*G\big((\xb,\yba),(\aba,\bb)\big)=\Big(\gph D^*G\big((\xb,\yba),(\aba,\bb)\big)\Big)^*\\
  &=\Big\{\big((u,v^*), (\nabla^2_{xx}\lag(\xb,\yba)u+\nabla F(\xb)^Tv^*, -\nabla F(\xb)u+w)\big)\mv (w,v^*)\in \bar L\Big\},\end{align*}
  where the first equation holds true by \cite[Equation (2.4)]{Gfr25c}, the second one follows from \eqref{EqLemStrProto} and the last one can be obtained by some elementary calculations from \eqref{EqAuxCoderivG}.
\end{proof}
\begin{remark}
\begin{enumerate}
\item The proof of the proposition above relies on the fact, that $\Sp(\partial g)(\yb,\yba)$ is a singleton if and only if $\partial g$ is strictly proto-differentiable at $\yb$ for $\yba$, cf. \cite[Corollary 3.3]{Gfr25c}.
\item For the class of {\em reliably $C^2$-decomposable functions} $g$ it was shown by Hang and Sarabi \cite[Theorem 5.8]{HaSa25} that $\partial g$ is strictly proto-differentiable at $\yb$ for $\yba$ if and only if $\yba\in\ri\partial g(\yb)$.
\item If $\gph \partial g$ is SCD semismooth* in the sense of \cite[Definition 5.1]{GfrOut22}, which is, e.g., the case when $\partial g$ is a closed subanalytic set, then for almost all $(y,y^*)\in \gph \partial g$ (with respect to the $m$-dimensional Hausdorff measure) the subgradient mapping $\partial g$ is strictly proto-differentiable at $y$ for $y^*$, cf. \cite[Corollary 4.3]{GfrOut25_prepr}.
\end{enumerate}
\end{remark}

Let us proceed with the general case . Recall the following definitions due to Rockafellar \cite{Ro19,Ro23a}.
\begin{definition}
\begin{enumerate}
  \item The \lsc function $\psi:\R^s\to\oR$  is called {\em variationally convex} at  $\zb\in \dom  \psi$ for $\zba\in\partial  \psi(\zb)$ if there is  some open convex neighborhood $\U\times \V^*$ of $(\zb,\zba)$, an \lsc function $\widehat \psi:\R^s\to\oR$ and a real  $\epsilon>0$ such that  $\widehat \psi$ is convex on $\U$, $\widehat \psi \leq  \psi$ on $\U$  and
\[\gph\partial  \psi\cap (\U_\epsilon\times \V^*) = \gph \partial \widehat \psi\cap (\U\times \V^*) \quad \mbox{and $\psi(z)=\widehat \psi(z)$ at the common elements $(z,z^*)$},\]
where $\U_\epsilon:=\{z\in \U\mv \psi(z)<\psi(\zb)+\epsilon\}$.\\
If this holds with  $\widehat \psi$ being strongly convex on $\U$, we will call $\psi$ to be {\em variationally strongly convex} at $\zb$ for $\zba$.
\item The {\em variational sufficient condition} holds for the problem \eqref{EqOptProbl} at $\xb$ for $\yba$ if there is some $\sigma\geq 0$ such that the function $\psi_\sigma:\R^n\times\R^m \to\oR$ given by
    \[\psi_\sigma(x,b):=f(x)+g(F(x)+b)+\frac \sigma2\norm{b}^2\]
    is variationally convex at $(\xb,\bb)$ for $(\aba,\yba)$. The {\em strong variational sufficient
condition} holds if $\psi_\sigma$ is variationally strongly convex.
\end{enumerate}
\end{definition}

A characterization of strong variational sufficiency in terms of quadratic bundles was given by Rockafellar in {\cite[Theorem 5]{Ro23a}}.
In order to facilitate the presentation we introduce the following notation.
\begin{definition}
  Given an \lsc function $\psi:\R^s\to\oR$ and a pair $(\zb,\zba)\in\gph \partial \psi$ such that $\psi$ is prox-regular and subdifferentially continuous at $\zb$ for $\zba$, we define the {\em q-subderivative} of $\psi$ at $\zb$ for $\zba$ as the function ${\rm d}^2_q \psi(\zb,\zba):\R^s\to\oR$ given by
    \begin{equation}\label{Eq_q_subderiv}
    {\rm d}^2_q \psi(\zb,\zba)(w)= 2\min\{q(w)\mv q\in{\rm quad\,}\psi(\zb,\zba)\},\ w\in\R^s.\end{equation}
\end{definition}
In view of \eqref{EqQuad_PW_Equiv} we also have the representation
\begin{equation}\label{Eq_q_subderiv_SCD}
{\rm d}^2_q \psi(\zb,\zba)(w)=\min\{\skalp{w,Ww}\mv (P,W)\in\M_{P,W}\partial \psi(\zb,\zba), w\in\rge P\}.
\end{equation}
Since the quadratic bundle ${\rm quad\,}\psi(\zb,\zba)$ is nonempty and compact with respect to epi-convergence, cf. \cite{Ro25_prepr},  and $\Sp(\partial \psi)(\zb,\zba)$ is compact in $\Z_s$, the minima in \eqref{Eq_q_subderiv} and \eqref{Eq_q_subderiv_SCD} are actually attained. Further, ${\rm d}^2_q \psi(\zb,\zba)$ is proper \lsc and
\begin{equation}\label{Eq_q_subderiv_pos}
\Big({\rm d}^2_q \psi(\zb,\zba)(w)>0\ \forall w\not=0\Big)\ \Longleftrightarrow\ \Big(\exists \mu>0:\ {\rm d}^2_q \psi(\zb,\zba)(w)\geq\mu\norm{z}^2\ \forall w\Big).\end{equation}
\begin{theorem}[{\cite[Theorem 5]{Ro23a}}]The strong variational  sufficient condition for the problem \eqref{EqOptProbl} holds at $\xb$ for $\yba$ if and only if
\begin{equation}\label{EqCharVarSuff}
  \skalp{u,\nabla^2_{xx}\lag(\xb,\yba)u}+{\rm d}^2_q g(F(\xb),\yba)(\nabla F(\xb)u)>0\ \forall u\not=0.
\end{equation}
\end{theorem}
If the variational sufficiency condition holds at $\xb$ for $\yba$, it is easy to see that for any $\epsilon>0$ the mapping $(x,b)\to\psi_{\sigma+\epsilon}(x,b)+\frac\epsilon2\norm{x}^2$ is variationally strongly convex and therefore the strong variational sufficiency condition holds for the problem \eqref{EqOptProbl} with $f$ replaced by $f+\frac\epsilon2\norm{\cdot}^2$. Hence
\[    \skalp{u,\nabla^2_{xx}\lag(\xb,\yba)u}+\epsilon\norm{u}^2+{\rm d}^2_q g(F(\xb),\yba)(\nabla F(\xb)u)>0\ \forall u\not=0\]
and, by  passing $\epsilon\downarrow0$, we obtain that the condition
\begin{equation}
  \label{EqNecVarSuff}  \skalp{u,\nabla^2_{xx}\lag(\xb,\yba)u}+{\rm d}^2_q g(F(\xb),\yba)(\nabla F(\xb)u)\geq 0\ \forall u.
\end{equation}
is necessary for variational sufficiency.

We are now in the position to formulate the second main result of this section.
\begin{theorem}\label{ThAubin2}
The following statements are equivalent.
  \begin{enumerate}
    \item[(i)] $\SKKT$ has the Aubin property around $\big((\aba,\bb),(\xb,\yba)\big)$ and the necessary condition for variational sufficiency \eqref{EqNecVarSuff} holds.
    \item[(ii)] Both SOQC and strong variational sufficiency  hold at $\xb$ for $\yba$.
    \item[(iii)] For every pair $(P,W)\in \M_{P,W}\partial g(F(\xb),\yba)$ there holds
    \begin{align}&\ker \nabla F(\xb)^T\cap(\rge P)^\perp=\{0\}\\
    &\bskalp{u,\big(\nabla^2_{xx}\lag(\xb,\yba)+\nabla F(\xb)^TW\nabla F(\xb)\big)u}>0\ \forall u\not=0: \nabla F(\xb)u\in\rge P.
    \end{align}
    \item[(iv)] For every $q\in {\rm quad\,}g(F(\xb),\yba)$ there holds
    \begin{align}
    &\ker \nabla F(\xb)^T\cap(\dom q)^\perp=\{0\} ,\\
      &\frac 12\skalp{u,\nabla^2_{xx}\lag(\xb,\yba)u}+q(\nabla F(\xb)u)>0\ \forall u\not=0.
    \end{align}
    \item[(v)]The variational sufficient condition holds at $\xb$ for $\yba$ and the primal--dual pair $(\xb,\yba)$ is fully stable in problem \eqref{EqOptProbl}.
    \item[(vi)] $\SKKT$ has a graphical localization around $\big((\aba,\bb),(\xb,\yba)\big)$ which is single-valued and Lipschitz continuous and the necessary condition for variational sufficiency \eqref{EqNecVarSuff} holds.
  \end{enumerate}
\end{theorem}
\begin{proof}
{\bf (i)$\Rightarrow$(ii):} If (i) holds, Proposition \ref{PropAubinSOQC} tells us that  SOQC is fulfilled at $\xb$ for $\yba$. We now will show by contraposition that the criterion \eqref{EqCharVarSuff} for strong variational sufficiency must hold  as well. Assume on the contrary that there is some $u\not=0$ with $\skalp{u,\nabla^2_{xx}\lag(\xb,\yba)u}+ d^2_q g(\yb,\yba)(\nabla F(\xb)u)\leq 0$. From \eqref{EqNecVarSuff} and \eqref{Eq_q_subderiv_SCD} we deduce that there exists $(P,W)\in\M_{P,W}\partial g(\yb,\yba)$ satisfying $\nabla F(\xb)u\in\rge P$ and
\[\bskalp{u,\big(\nabla^2_{xx}\lag(\xb,\yba)+\nabla F(\xb)^TW\nabla F(\xb)\big)u}=0.\]
Since the necessary condition for variational sufficiency \eqref{EqNecVarSuff} tells us that the symmetric matrix
\[\nabla^2_{xx}\lag(\xb,\yba)+\nabla F(\xb)^TW\nabla F(\xb)\]
is positive semidefinite on the subspace $U:=\{s\mv \nabla F(\xb)s\in\rge P\} =\ker (I-P)\nabla F(\xb)$, we conclude that
\[\big(\nabla^2_{xx}\lag(\xb,\yba)+\nabla F(\xb)^TW\nabla F(\xb)\big)u\in U^\perp=\rge \nabla F(\xb)^T(I-P)\]
and consequently there is some $p$ satisfying
\[\nabla F(\xb)^T(I-P)p= \nabla F(\xb)^TW(I-P)p = \big(\nabla^2_{xx}\lag(\xb,\yba)+\nabla F(\xb)^TW\nabla F(\xb)\big)u.\]
By setting $\bar p=\nabla F(\xb)u-(I-P)p$ and taking into account \eqref{EqIncl1}, it follows that
\[(P\bar p, W\bar p)=(\nabla F(\xb)u, W\bar p)\in\rge(P,W)\subset \gph D^*(\partial g)(\yb,\yba)\]
and
\[\nabla^2_{xx}\lag(\xb,\yba)u+\nabla F(\xb)^TW\bar p=\big(\nabla^2_{xx}\lag(\xb,\yba)+\nabla F(\xb)^TW\nabla F(\xb)\big)u-\nabla F(\xb)^TW(I-P)p=0.\]
Since $u\not=0$, the  criterion \eqref{EqMord} is violated for $u$ and $v^*=W\bar p$, contradicting the assumed Aubin property of $\SKKT$ around $\big((\aba,\bb),(\xb,\yba)\big)$. Thus, the criterion \eqref{EqCharVarSuff}, which is equivalent to strong variational sufficiency, holds true and the implication (i)$\Rightarrow$(ii) is established.

{\bf (ii)$\Leftrightarrow$(iii), (ii)$\Leftrightarrow$(iv):} These equivalences hold true by Remark \ref{RemEquiv_SOQC_SCD}, the characterization of strong variational sufficiency by  \eqref{EqCharVarSuff} and \eqref{Eq_q_subderiv}, \eqref{Eq_q_subderiv_SCD}.

 {\bf (ii)$\Rightarrow$(v):} The implication follows immediately from \cite[Theorem 4.2]{BeRo24} by observing that \cite[Equation (4.17)]{BeRo24} amounts to the condition \eqref{EqSOQC_StrictDer} defining SOQC. Note that we may apply \cite[Theorem 4.2]{BeRo24} because SOQC implies the basic constraint qualification \eqref{EqBQC}.

 {\bf (v)$\Rightarrow$(vi):} The proof of this implication relies on \cite[Theorem 3.4]{BeRo24}. By carefully checking the proof of \cite[Theorem 3.4]{BeRo24} we see that it does not use the basic constraint qualification \cite[Equation (1.3)]{BeRo24} and therefore we need not to assume that \eqref{EqBQC} holds.\\
 Consider $\delta>0$ and the neighborhoods $\U^*\times \V$ of $(\aba,\bb)$ according to Definition \ref{DefPrimalDualStability} such that $\widehat S_{\rm Opt}^\delta(a^*,b)$ is single-valued and Lipschitz on $\U^*\times \V$. By using \cite[Theorem 3.4]{BeRo24} there are neighborhoods $\tilde\U^*\times \tilde \V\times\X\times \Y^*$ of $\big((\aba,\bb),(\xb,\yba)\big)$ such that for every $\big((a^*,b),(x,y^*)\big)\in\gph \SKKT\cap (\tilde\U^*\times \tilde \V\times\X\times \Y^*)$ the point $x$ minimizes $\varphi_{a^*,b}:=f(\cdot)-\skalp{a^*,\cdot}+g(F(\cdot)+b)$ over $\X$. Next we choose $\hat\delta\leq\delta$ satisfying $\B(\xb,\hat\delta)\times \B(\yba,\hat\delta)\subset \X\times \Y^*$ together with neighborhoods $\hat\U^*\times\hat \V$ of $(\aba,\bb)$ fulfilling
 $\widehat S_{\rm Opt}^\delta(a^*,b)\in \B(\xb,\hat\delta)\times \B(\yba,\hat\delta)$ $\forall (a^*,b)\in\hat\U^*\times \hat \V$, implying that $\widehat S_{\rm Opt}^\delta$ coincides with $\widehat S_{\rm Opt}^{\hat\delta}$ on $\hat\U^*\times \hat \V$.  Since for every $(a^*,b)\in(\tilde \U^*\cap\hat\U^*)\times(\tilde \V\cap\hat \V)$ and every $(x,y^*)\in \SKKT(a^*,b)\cap(\B(\xb,\hat\delta)\times \B(\yba,\hat\delta))$ the point $x\in \B(\xb,\hat\delta)$ minimizes $\varphi_{a^*,b}$ over $\X\supset \B(\xb,\hat\delta)$, we obtain that
 \begin{align*}&\lefteqn{\gph \SKKT\cap\big((\tilde \U^*\cap\hat\U^*)\times(\tilde \V\cap\hat \V)\times\B(\xb,\hat\delta)\times \B(\yba,\hat\delta))}\\
 &= \gph \widehat S_{\rm Opt}^{\hat\delta} \cap  \big((\tilde \U^*\cap\hat\U^*)\times(\tilde \V\cap\hat \V)\times\B(\xb,\hat\delta)\times \B(\yba,\hat\delta)\big)\\
 &= \gph \widehat S_{\rm Opt}^\delta \cap  \big((\tilde \U^*\cap\hat\U^*)\times(\tilde \V\cap\hat \V)\times \R^n\times\R^m\big)
 \end{align*}
 and therefore $\widehat S_{\rm Opt}^\delta$ provides a single-valued Lipschitzian graphical localization for $\SKKT$.

{\bf (vi)$\Rightarrow$(i):} This holds trivially true.
\end{proof}
Let us illustrate this theorem by the following example.
\begin{example}\label{ExPoly}
  Let $g=\delta_C$ be the indicator function of a convex polyhedral set $C\subset\R^m$. By \cite[Example 3.29]{GfrOut22}, $\Sp(\partial\delta_C)(\yb,\yba)$ is the collection of all subspaces $(\F-\F)\times(\F-\F)^\perp$, where $\F$ is a face of the critical cone $\K_C(\yb,\yba)=\{v\in T_C(\yb)\mv \skalp{\yba,v}=0\}$. Clearly, for every such subspace we have $\{v^*\mv (0,v^*)\in (\F-\F)\times(\F-\F)^\perp\}=(\F-\F)^\perp$ and, together with the fact, that the lineality space $\lin T_C(\yb):=T_C(\yb)\cap(-T_C(\yb))$ is the smallest face of the critical cone, we obtain that SOQC amounts to $\ker\nabla F(\xb)\cap (\lin T_C(\yb))^\perp=\{0\}$, which is equivalent to the co-called {\em nondegeneracy condition}
  \[\rge \nabla F(\xb)+\lin T_C(\yb)=\R^m,\]
  cf. \cite[Equation (4.172)]{BonSh00}.
  Further, since for every $(v,v^*)\in (\F-\F)\times(\F-\F)^\perp$ we have $\skalp{v,v^*}=0$ and the largest face of the critical cone is the crical cone itself, we obtain that ${\rm d}^2_q\delta_C(\yb,\yba)=\delta_{\K_C(\yb,\yba)-\K_C(\yb,\yba)}$. Hence, the strong variational sufficient condition  amounts to the {\em strong second-order sufficient condition}
  \[\skalp{u,\nabla^2_{xx}\lag(\xb,\yba)u}>0\ \forall u\not=0:\nabla F(\xb)u\in\big(\K_C(\yb,\yba)-\K_C(\yb,\yba)\big),\]
  cf. \cite[Example 2, Theorem 4]{Ro23a}.

Recall that the graph of the limiting coderivative $\gph D^*(\partial\delta_C)(\yb,\yba)$ is the collection of all cones of the form
  $(\F_1 -\F_2)\times(\F_1-\F_2)^\circ$, where $\F_1,\F_2$ are two faces of the critical cone $\K_C(\yb,\yba)$ satisfying $\F_2\subset \F_1$, cf. \cite{DoRo96}.  We see that the the SC derivative has a simpler structure than the limiting coderivative. Further note that the structure of the   strict graphical derivative is even more involved, cf. \cite[Proposition 4H10]{DoRo14}.
\end{example}

The natural question arises whether strong variational sufficiency in problem \eqref{EqOptProbl} is somehow related with variational convexity of of the objective $\varphi$. Assume that the basic constraint qualification \eqref{EqBQC} holds ensuring that for every $(a^*,b)$ sufficiently close to $(\aba,\bb)$, every $\delta>0$ sufficiently small and every $x\in S_{\rm Opt}^\delta(a^*,b)$ the multiplier set $\{y^*\mv (x,y^*)\in\SKKT(a^*,b)\}$ is nonempty.
If the primal-dual  pair $(\xb,\yba)$ is fully stable in problem  \eqref{EqOptProbl}, it  easily follows from the definition that $\xb$ is a tilt-stable minimizer for the function $\varphi$ given by \eqref{EqOptProbl}.
\begin{definition}\label{DefTiltStab} Let $\varphi :\R^n\to\oR$ be a function and let $\xb\in\dom \varphi$. Then
 $\xb$ is a {\em tilt-stable local minimizer} of $\varphi$ if there is a number $\delta>0$ such that the mapping
\begin{equation}\label{EqM_gamma}
\tilde S_{\rm Opt}^\delta(a^\ast):=\argmin\big\{\varphi(x)-\skalp{a^\ast,x}\mv x\in\B(\xb,\delta)\big\},\quad a^*\in\R^n,
\end{equation}
is single-valued and Lipschitz continuous in some neighborhood of $\aba=0$ with $\tilde S_{\rm Opt}^\delta(0)=\{\xb\}$.
\end{definition}
Tilt-stability of $\xb$  is equivalent to variational strong convexity of $\varphi$ at $\xb$ for $0$, see, e.g., \cite{Gfr25b}, and, in the setting of \eqref{EqOptProbl}, it can be characterized  by the condition
\begin{equation*}
   \frac 12\skalp{u,\nabla^2 f(\xb)u}+\tilde q(u)\geq\mu\norm{u}^2\ \forall \tilde q\in {\rm quad\,}(g\circ F)(\xb,\xba)\ \forall u
\end{equation*}
with some $\mu>0$ by \cite[Proposition 3.10, Theorem 4.4]{KhMoPhVi26}. By using the q-subderivative of $g\circ F$, we conclude from \eqref{Eq_q_subderiv_pos} that strong variational convexity  and tilt-stability can be characterized by the condition
\begin{equation}\label{EqVarStrConv2}
  \skalp{u,\nabla^2 f(\xb)u}+{\rm d}^2_q(g\circ F)(\xb,\xba)(u)>0\  \forall u\not=0.
\end{equation}
Comparing \eqref{EqCharVarSuff} and  \eqref{EqVarStrConv2} suggests the following definition.
\begin{definition}\label{DefChRule}
 We say that the {\em  chain rule for the q-subderivative} of the composite function $g\circ F$ holds at $\xb$ for $\yba$ if  one has
   \begin{align}
   {\rm d}^2_q(g\circ F)(\xb,\nabla F(\xb)^T\yba)(u)= \skalp{u,\nabla^2\skalp{\yba,F}(\xb)u} +{\rm d}^2_q g(F(\xb),\yba)(\nabla F(\xb)u)\ \forall u\in\R^n,
   \end{align}
   provided that SOQC holds at $\xb$ for $\yba$.
 \end{definition}
 Our considerations above yield the following statement.
 \begin{theorem}In addition to our standing assumptions, suppose that  the basic constraint qualification \eqref{EqBQC} is fulfilled
   and the chain rule for the q-subderivative of $g\circ F$ holds at $\xb$ for $\yba$. Then the following statements  are equivalent.
 \begin{enumerate}
   \item[(i)] $\SKKT$ has the Aubin property around $\big((\aba,\bb),(\xb,\yba)\big)$ and the necessary condition for variational sufficiency \eqref{EqNecVarSuff} holds.
    \item[(ii)] Both SOQC and strong variational sufficiency  hold at $\xb$ for $\yba$.
    \item[(iii)] For every pair $(P,W)\in \M_{P,W}\partial g(F(\xb),\yba)$ there holds
    \begin{align}&\ker \nabla F(\xb)^T\cap(\rge P)^\perp=\{0\}\\
    &\bskalp{u,\big(\nabla^2_{xx}\lag(\xb,\yba)+\nabla F(\xb)^TW\nabla F(\xb)\big)u}>0\ \forall u\not=0: \nabla F(\xb)u\in\rge P.
    \end{align}
    \item[(iv)] For every $q\in {\rm quad\,}g(F(\xb),\yba)$ there holds
    \begin{align}
    &\ker \nabla F(\xb)^T\cap(\dom q)^\perp=\{0\} ,\\
      &\frac 12\skalp{u,\nabla^2_{xx}\lag(\xb,\yba)u}+q(\nabla F(\xb)u)>0\ \forall u\not=0.
    \end{align}
    \item[(v)] The variational sufficient condition holds at $\xb$ for $\yba$ and the primal--dual pair $(\xb,\yba)$ is fully stable in problem \eqref{EqOptProbl}.
    \item [(vi)] The primal--dual pair $(\xb,\yba)$ is fully stable in problem \eqref{EqOptProbl}.
    \item[(vii)] SOQC holds at $\xb$ for $\yba$ and $\xb$ is a tilt-stable local minimizer for \eqref{EqOptProbl}.
    \item[(viii)] $\SKKT$ has a graphical localization around $\big((\aba,\bb),(\xb,\yba)\big)$ which is single-valued and Lipschitz continuous and the necessary condition for variational sufficiency \eqref{EqNecVarSuff} holds.
  \end{enumerate}
\end{theorem}
 \begin{proof}In view of Theorem \ref{ThAubin2} it suffices to proof the equivalences (v)$\Leftrightarrow$(vi)$\Leftrightarrow$(vii).

 {\bf (v)$\Rightarrow$(vi): } This holds trivially true.

 {\bf (vi)$\Rightarrow$(vii): } SOQC follows from full primal-dual stability of $(\xb,\yba)$ by
 \cite[Theorem 1.5]{BeRo24}, since condition c) of
    \cite[Theorem 1.5]{BeRo24} amounts in our setting to \eqref{EqSOQC_StrictDer} defining SOQC. Tilt-stability of $\xb$ follows from the definition of full primal-dual stability.

 {\bf (vii)$\Rightarrow$(v): } (vii) implies (ii) because the characterization \eqref{EqVarStrConv2} for tilt-stability of $\xb$ together with the imposed chain rule for the q-subderivative implies the characterization \eqref{EqCharVarSuff} of strong variational sufficiency. The implication (ii)$\Rightarrow$(v) was already established in Theorem \ref{ThAubin2}.
 \end{proof}
 It is easy to see that the chain rule for the SC derivative \eqref{EqChainRuleSCD} implies the chain rule for the q-subderivative. Hence, whenever $\nabla F(\xb)$ has full row rank or $\Sp (\partial g)(\yb,\yba)$ is a singleton, it follows from Proposition \ref{PropChainRuleSCD} that the chain rule for the q-subderivative is valid.
 It is beyond the scope of the paper to investigate more advanced situations where the chain rules for the SC-derivative and q-subderivative, respectively, hold true.

 \section{Conclusion}
The paper has been devoted to a deep analysis of Lipschitzian stability of the KKT mapping $\SKKT$ in context of
a composite optimization problem. To this aim quite a broad arsenal of tools from modern variational analysis has been
employed which includes, apart from standard constructions, also the so-called quadratic bundles and SC generalized
derivatives. The main results characterize the investigated  stability by several equivalent statements  which are
valid under various circumstances. Concretely, one considers the following three situations when
\begin{enumerate}
\item[(i)] the SC derivative of $\partial g$ at the reference point is a singleton;
\item[(ii)] either a necessary or a sufficient condition related to strong variational sufficiency is fulfilled;
\item[(iii)] a special chain rule for the q-subderivatives of the composition $g\circ F$ is valid.
\end{enumerate}

\bibliographystyle{acm}
\bibliography{../../gfrerer}

\end{document}